\documentclass[11pt]{amsart}

\usepackage{amssymb}
\usepackage{amsmath}
\usepackage{amsthm}
\usepackage{graphicx}
\usepackage{macros}
\usepackage[usenames,dvipsnames]{xcolor}
\usepackage{tikz}
\usepackage{subcaption}
\usepackage{tikz}
\usepackage{float}
\usepackage{color}
\usepackage{hyperref}
\hypersetup{
  colorlinks,
  citecolor= Brown,
  linkcolor= link,
  urlcolor= link
}

\usepackage{appendix}
\usepackage{array}
\usepackage{footnote}
\usepackage{booktabs}
\usepackage{hhline}
\makesavenoteenv{tabular}
\usepackage{scrextend}

 \deffootnote{0em}{1.6em}{\enskip}

\definecolor{link}{rgb}{0.18,0.25,0.63}
\definecolor{myred}{rgb}{0.7,0.25,0.2}
\usepackage[top=3cm,bottom=3cm,left=2.5cm,right=2.5cm]{geometry}

\numberwithin{equation}{section}

\makeatletter
\g@addto@macro{\endabstract}{\@setabstract}
\newcommand{\authorfootnotes}{\renewcommand\thefootnote{\@fnsymbol\c@footnote}}%
\makeatother

\allowdisplaybreaks

\newtheorem{prop}{Proposition}[section]
\newtheorem*{prop*}{Proposition}
\newtheorem{rem}[prop]{Remark}
\newtheorem*{rem*}{Remark}
\newtheorem{thm}[prop]{Theorem}
\newtheorem*{thm*}{Theorem}
\newtheorem{defn}[prop]{Definition}
\newtheorem*{defn*}{Definition}
\newtheorem{lem}[prop]{Lemma}
\newtheorem*{lem*}{Lemma}
\newtheorem{cor}[prop]{Corollary}
\newtheorem*{cor*}{Corollary}

\DeclareMathOperator{\TV}{TV}

\DeclareMathOperator{\BV}{BV}

\newcommand{\dkl}{D_{\text{KL}}}

\DeclareMathOperator{\range}{Rg}
\DeclareMathOperator{\dive}{div}
\DeclareMathOperator{\domain}{\operatorname{dom}}

\DeclareMathOperator{\proj}{proj}
\DeclareMathOperator{\prox}{prox}

\newcommand{\ds}{{d^*}}

\newcommand{\G}{ \mathcal{G} }
\newcommand{\Hc}{ \mathcal{H} }

\newcommand{\R}{\mathbb{R}}
\newcommand{\N}{\mathbb{N}}

\newcommand{\Du}{\mathrm{D} u }

\newcommand{\beq}{\begin{equation}}
\newcommand{\eeq}{\end{equation}}

\newcommand{\I}{\mathcal{I}}

\begin{document}

\definecolor{link}{rgb}{0,0,0}
\definecolor{mygrey}{rgb}{0.34,0.34,0.34}
\def\blue #1{{\color{blue}#1}}

 \begin{center}
 \large
  \textbf{A function space framework for structural total variation regularization with applications in inverse problems} \par \bigskip \bigskip
   \normalsize
  \textsc{Michael Hinterm\"uller}\textsuperscript{$\dagger\ddagger$}, \textsc{Martin Holler} \textsuperscript{$\mathsection$} \textsc{and Kostas Papafitsoros}  \textsuperscript{$\dagger$}
\let\thefootnote\relax\footnote{
\textsuperscript{$\dagger$}Weierstrass Institute for Applied Analysis and Stochastics (WIAS), Mohrenstrasse 39, 10117, Berlin, Germany
}
\let\thefootnote\relax\footnote{
\textsuperscript{$\ddagger$}Institute for Mathematics, Humboldt University of Berlin, Unter den Linden 6, 10099, Berlin, Germany}
\let\thefootnote\relax\footnote{
\textsuperscript{$\mathsection$}Institute for Mathematics and Scientific Computing, University of Graz, Heinrichstrasse 36, A-8010, Graz, Austria. The institute is a member of NAWI Graz (\texttt{www.nawigraz.at}). M.~H. is a member of BioTechMed Graz (\texttt{www.biotechmed.at}).}

\let\thefootnote\relax\footnote{
\hspace{3.2pt}Emails: \href{mailto:Michael.Hintermueller@wias-berlin.de}{\nolinkurl{Michael.Hintermueller@wias-berlin.de}}}
 \let\thefootnote\relax\footnote{
 \hspace{37pt}\href{mailto: Martin.Holler@uni-graz.at}{\nolinkurl{Martin.Holler@uni-graz.at}}
 }
 \let\thefootnote\relax\footnote{
 \hspace{37pt}\href{mailto:Kostas.Papafitsoros@wias-berlin.de}{\nolinkurl{Kostas.Papafitsoros@wias-berlin.de}}
 }

\end{center}

\begin{abstract}
In this work, we introduce a function space setting for a wide class of structural/weighted total variation (TV) regularization methods motivated by their applications in inverse problems. In particular, we consider a regularizer that is the appropriate lower semi-continuous envelope (relaxation) of a suitable total variation type functional initially defined for sufficiently smooth functions. We study examples where this relaxation can be expressed explicitly, and we also provide refinements for weighted total variation for a wide range of weights. Since an integral characterization of the relaxation in function space is, in general, not always available, we show that, for a rather general linear inverse problems setting, instead of the classical Tikhonov regularization problem, one can equivalently solve a saddle-point problem where no a priori knowledge of an explicit formulation of the structural TV functional is needed. In particular, motivated by concrete applications, we deduce corresponding results for linear inverse problems with norm and Poisson log-likelihood data discrepancy terms.
Finally, we  provide proof-of-concept numerical examples where we solve the saddle-point problem for weighted TV denoising as well as for MR guided PET image reconstruction.
\end{abstract}


%


\definecolor{link}{rgb}{0.18,0.25,0.63}


\section{Introduction}
In many classical applications of inverse problems, rather than measuring data from a single channel, recently the simultaneous acquisition from multiple channels has gained importance. 
Besides having different sources of information available, a main advantage of such multiple measurements is due to the possibility of exploiting correlations between different data channels in the inversion process, often leading to a significant improvement for each individual channel.
In particular, when the underlying quantities of interest can be visualized as image data, these correlations typically correspond to joint structures in images. In view of this background, multimodality and multicontrast imaging explore such joint structures for improved reconstruction. Successful applications of these techniques can be found for instance in biomedical imaging \cite{bathke2017improved, Ehrhardt_structure_TV, Ehrhardt_PET_prior, Ehrhardt_PET_MRI,Holler_PET_MRI,Rasch17,Rigie2015,Schramm17_pls, Vunckx2012}, geosciences \cite{Steklova2017}, electron microscopy \cite{haberfehlner2014nanoscale} and several more.

In this context, one distinguishes two different approaches for exploiting correlations: (i) Joint reconstruction techniques that treat all available channels equally, such as \cite{Holler_PET_MRI}, and (ii) structural-prior-based regularization techniques that assume some ground truth structural information to be available.
Here we focus on a particular class of type (ii), namely {\it structural total-variation-type regularization functionals}, i.e., functionals which integrate a spatially-dependent pointwise function of the image gradient for regularization. In this vein, given some prior information $v$, we consider a regularization functional of the type
\[J(u)=\int_{\om} j_{v}(x,\nabla u(x)) \, dx,\]
which is used in a Tikhonov-type regularization problem where $u$ is reconstructed by solving
\begin{equation}\label{intro_general_primal}
\min_{u} \int_{\om} j_{v}(x,\nabla u(x)) \, dx + (\lambda D\circ K)(u).
\end{equation}
Here, $D$ denotes a data discrepancy term and $K$ a bounded linear operator which reflects the forward model. A very basic, formal example of such a regularization is given when $v$ denotes the underlying ground truth image and we incorporate information on its gradient by defining $j_v(x,z) := \frac{1}{|\nabla v(x)|} |z|$, where $|\cdot|$ is some norm. This yields 
\[J(u)=\int_{\om} \frac{1}{|\nabla v(x)|} |\nabla u(x)| \, dx,\]
which corresponds to a weighted total variation (TV) functional. 

Even though problems of the form \eqref{intro_general_primal} limit the exchange of information to the gradient level, they cover a large set of existing works in particular in the context of medical image processing. This can be explained on the one hand by the popularity of TV-type regularization for imaging in general, which is due to its capability of recovering jump discontinuities. On the other hand, the gradient level seems very well suited for exploiting correlations between different contrasts or modalities as it primarily encodes structural information, i.e., to some extent it is independent of the absolute magnitude of the signal. 
With the goal of enforcing parallel level sets, in \cite{Ehrhardt_PLV}, for instance, the regularizer $J$ was chosen for a given image $v$ as
\begin{equation}\label{intro_funct_1}
J(u)=\int_{\om}\phi\left (\psi(|\nabla v|_{\beta} |\nabla v|_{\beta})-\psi (|(\nabla u\cdot \nabla v)|_{\beta^{2}}) \right )\,dx,
\end{equation}
where $\phi, \psi$ are appropriate increasing functions, $|\cdot|_{\beta}, \, |\cdot |_{\beta^ 2}$ denote smoothed norms, and $a\cdot b$ is the scalar (``dot'') product of vectors $a,b\in\mathbb{R}^d$. This definition is motivated by the fact that for two vectors $w,y$ the quantity $|w||y|-|(w,y)|$ is zero if and only if they are parallel to each other. Particular instances of \eqref{intro_funct_1} were then also used in \cite{Ehrhardt_PET_MRI} for joint MR-PET reconstruction.
A similar functional, but in the spirit of \cite{Kaipio_structural} is
\begin{equation}\label{intro_funct_2}
J(u)=\int_{\om} |\nabla u|^{2}-(\nabla u \cdot w)^{2}\,dx.
\end{equation}
Here $w$ denotes some a priori vector field that contains gradient information.
In the context of multicontrast MR, the authors in \cite{Ehrhardt_structure_TV} choose
\begin{equation}\label{intro_funct_3}
J(u)=\int_{\om}\left | \left (I -\frac{\nabla v \otimes \nabla v}{|\nabla v|^{2}} \right )\nabla u \right | \,dx=\int_{\om} |\nabla u| \sin \theta \,dx. \end{equation}
where again a smoothed version of the absolute value was used in the denominator. Here, $\theta$ denotes the angle between $\nabla u$ and $\nabla v$.
Observe that the latter functional can also be written as 
\[\int_{\om} \left (|\nabla u|^{2}-\left(\nabla u\cdot\frac{\nabla v}{|\nabla v|}\right)^{2} \right )^{1/2},\]
bearing similarity to the functional \eqref{intro_funct_2}.

We note that the regularization approaches above have been considered only in discrete settings, despite the original continuous formulations. Concerning the latter and as indicated above in connection with \eqref{intro_general_primal}, it is natural to consider $u\in \bv(\om)$, i.e., $|u|$ is Lebesgue integrable and $u$ is of bounded total variation, still allowing for jump discontinuities, i.e., sharp edges.
As a consequence, the (generalized) gradient $Du$ of $u$ is in general a finite Radon measure, only.
%
This fact, however, challenges the proper defininition of regularization functionals of the above types in an associated function space setting. The first part of our work aims at addressing precisely this issue. In fact, resorting to the concept of functions of a measure \cite{aubert2006mathematical} we propose to use the convex biconjugate as a regularizer. Indeed, starting from a general function $j:\om\times \RR^{d}\to [0,\infty)$, $\om\subset \RR^{d}$ with minimal assumptions, e.g., convexity, linear growth, $1$-homogeneity  in the second variable, we define the functional $J:L^{p}(\om)\to \RR$ as 
\begin{equation} \label{intro:eq:original_reg_func}
J(u) = \begin{cases}
\int _\Omega j(x,\nabla u(x) ) \,dx &\text{if } u \in W^{1,1}(\Omega), \\
\infty & \text{else,}
\end{cases}
\end{equation}
where we omit the dependence of $j$ on the prior $v$ for the sake of unburdening notation. Here, $L^p(\om)$ and $W^{1,1}(\om)$ denote the usual Lebesgue and Sobolev spaces; see \cite{adams_sobolev}.
We note that $J$ in \eqref{intro:eq:original_reg_func} is not suitable for variational regularization as it is finite only for a class of rather regular functions (i.e., in $W^{1,1}(\om)$) and it is not lower semi-continuous in an appropriate topology. As a remedy, we propose to resort to the convex biconjugate $J^{\ast\ast}$ of $J$, which we call {\it structural TV functional}. We recall that $J^{\ast\ast}$ coincides with the lower semi-continuous envelope (relaxation) of $J$ with respect to $L^{p}$ convergence. In general, $J^{\ast\ast}$ may not have an explicit representation, but we show that it can always be expressed in a dual form as
\begin{equation}\label{intro:J_astast_dual}
J^{\ast\ast}(u)=\sup_{g\in Q} \int_{\om} u\,\di g\,dx,\quad \text{for } u\in L^{p}(\om),
\end{equation}
where
\begin{equation}\label{intro:Q}
 Q := \left\{g \in W^q_0(\di;\om)\cap L^\infty(\Omega,\R^d): \;j^\circ(x,g(x))\leq 1 \text{ for almost every (a.e.) } x \in \Omega\right\},
 \end{equation}
see \eqref{mh:supp} below for the precise definition of the support function $j^{\circ}$, and we refer to \cite{Girault} for details on $W^q_0(\di;\om)$. Nevertheless, based on a recent result by Amar, De Cicco and Fusco \cite{Fusco08}, by linking $J^{\ast\ast}$ to the relaxed functional of $J$ with respect to $L^{1}$ convergence we are able to provide an integral  representation, under additional assumptions. For instance, in the case where $j(x,z)=\alpha(x)|z|$,  $\alpha\in \bv(\om)$, with $\alpha\ge 0$, $J^{\ast\ast}$ is equal to the weighted TV-functional
\[J^{\ast\ast}(u)=\int_{\om}\alpha^{-}\,d|Du|,\quad u\in \bv(\om),\]
with $\alpha^{-}$ the approximate lower limit of $\alpha$; see \eqref{mh:2} below for its definition. The set $Q$ then reads
\[ Q = \left\{g \in W^q_0(\di;\om)\cap L^\infty(\Omega,\R^d): \;|g(x)|\leq \alpha(x) \text{ for a.e. } x \in \Omega\right\}.
\]
Interestingly, as a consequence of this formulation, we get certain density results of convex intersections in the spirit of \cite{sing_mol, Hint_Rau_density, hint_rau_ros} as a byproduct; compare Section  \ref{sec:refine_weightedTV}.

Taking advantage of duality theory, in the second part of the paper we use the structural TV functional $J^{\ast\ast}$ for the regularization of linear inverse problems. In particular we study the general minimization problem
\begin{equation}\label{intro:primal}
\inf_{u\in L^{p}(\om)} J^{\ast\ast}(u)+ (\lambda D\circ K)(u).
\end{equation}
As emphasized above, an explicit representation of the functional $J^{\ast\ast}$ is available only under some additional, perhaps restrictive assumptions. In order to solve \eqref{intro:primal} without invoking such assumptions, we employ Fenchel-Rockafellar duality and show, in the continuous setting, equivalence of \eqref{intro:primal} to a saddle-point problem of the form
%

\begin{equation}\label{intro:saddle_point}
 \inf_{\substack{p \in W^q_0(\di;\om) \\ p \in Q}} \sup_{u\in L^p (\Omega)} (\di p ,u)  -  (\lambda D\circ K)(u),
 \end{equation}
 where $(\cdot ,\cdot)$ denotes an appropriate pairing.
This is achieved under assumptions which are tight in the sense that we can provide a counterexample where, without these assumptions, even existence of a solution for \eqref{intro:primal} fails.
The major advantage of the above saddle-point reformulation is that it allows to obtain a solution of the original problem without requiring an explicit form of neither $J^ {**}$ nor $(D \circ K)^ *$. Note that the latter would be required for solving the predual problem. Furthermore, it has a format directly amenable to duality-based numerical optimization algorithms.

The equivalence to a saddle-point reformulation is obtained under rather general assumptions on the data discrepancy term $D$, which, as corollary, allows us to cover the case of any norm discrepancy term as well as the case of a log-likelihood term for Poisson noise, which is relevant for instance for PET image reconstruction and electron tomography. The latter leads to the minimization problem 
 \begin{equation}\label{intro:eq:min_prob_pet_formal}
\min _{u \in L^p(\Omega) } J^{**}(u) + \lambda \int _\Sigma K u - f \log (Ku + c_0)\,d\sigma + \I_{[0,\infty)} (u).
\end{equation}
where $K$ denotes a Radon-transform-type operator, $f,\, c_0$ some given data, and $\I$ is the  indicator function for a set $M$, i.e., $\I_M(u) = 0$ if $u \in M$ and $\I_M(u) = +\infty $ otherwise.

 Finally, we show the versatility of our approach with proof-of-concept numerical examples in weighted TV denoising with vanishing weight function and MR guided PET reconstruction.
 
 We note here that there is previous work on the analysis of weighted and/or structural TV regularization in an infinite dimension setting.  In \cite{grasmair2010anisotropic}, another instance of structural-TV type functionals is employed, but the work only considers the case of image denoising. Further, the authors simultaneously optimize over the image data and an anisotropy in the TV-term, which leads to a non-convex problem. Regarding properties of solutions of weighted TV denoising we refer to the work by Jalalzai \cite{jalalzai2014discontinuities} as well as to \cite{mine_spatial}.  Finally, we mention that in \cite{novaga_weighted} the authors analyze a weighted TV regularization model for vortex density models. 

\subsection*{Structure of the paper}
The paper is organized as follows: In Section \ref{sec:notation}, we fix our notation and we remind the reader of basic facts concerning functions of bounded variation and $W(\mathrm{div})$ spaces. 

In Section \ref{sec:relaxation} we describe the relaxation framework for the structural TV functional. Under some conditions, we provide an integral representation of the relaxation based on a result of \cite{Fusco08} and we also provide a characterization of its subdifferential. 

Some refinements for weighted TV functional are given in Section \ref{sec:refine_weightedTV} in the case of continuous and lower semi-continuous weight functions. We show that the functional can be defined in a dual fashion, using smooth test functions and as a byproduct we obtain certain density results of convex intersections in the spirit of \cite{Hint_Rau_density, hint_rau_ros}. 

Section \ref{sec:duality} contains the main duality result of the paper. In particular, we show that under certain mild assumptions, the variational regularization problem with the relaxed structural TV functional as regularizer can be equivalently formulated as a saddle-point problem which requires knowledge neither of the explicit form of the relaxation--as it is the case for the primal problem --nor of the convex conjugate of the discrepancy term--as it is the case for the predual problem. As particular application, we elaborate on the case of Poisson log-likelihood discrepancy terms and show how the result can be transferred to this situation.

Finally, in Section \ref{sec:numerics} we provide proof-of-concept numerical examples, where we solve our saddle point problem for weighted TV denoising with vanishing weight function and for MR-guided PET image reconstruction.

\section{Notation and Preliminaries}\label{sec:notation}

\subsection{Functions of bounded variation}
Throughout the paper, $\om\subset\RR^{d}$, with $d\geq 2$, will be a bounded, open set with Lipschitz boundary, and we denote $d^{\ast}=\frac{d}{d-1}$. By $p,q$ we always denote two real numbers such that $p,q \in [1,\infty]$ and $q = \frac{p}{p-1}$ if $p\in (1,\infty)$, $q = \infty$ if $p=1$, and $q = 1$ if $p=\infty$. 

 The space of \emph{functions of bounded variation} on $\om$ is denoted by $\bv(\om)$. We have that $u\in\bv(\om)$ if and only if it is in $L^1(\Omega)$ and its distributional derivative is a bounded Radon measure, denoted by $Du$. The total variation $\tv(u)$ of $u$ is defined to be the total variation of  that measure, i.e.,  $\tv(u)=|Du|(\om)$ and it is equal to 
\begin{equation}\label{tv_def}
|Du|(\om)=\sup\left \{ \int_{\om} u\,\di\phi\,dx : \phi\in C_{c}^{\infty}(\om,\RR^{d}),\;|\phi(x)|\le 1,\; \forall x\in\om\right \}.
\end{equation}
The measure $Du$ can be decomposed into
\[Du=D^{a}u+D^{j}u + D^{c}u,\]
where $D^{a}u$ is the absolutely continuous with respect to Lebesgue measure $\mathcal{L}^{d}$, with density function denoted by $\nabla u$, $D^{j}u$ denotes the jump part which is the restriction to the jump set $J_{u}$ of $u$, and $D^{c}u$ is the Cantor part of $Du$. We recall that $D^j_u$ is defined over the set of points in $\om$ for which that $u^{+}(x)> u^{-}(x)$ where
\begin{align}
 u^+(x) &= \inf \left\{ t \in [-\infty,\infty] : \lim _{r\rightarrow 0} \frac{\mathcal{L}^{d}(\{v>t\}\cap B(x,r))}{r^d}=0\right\},\label{mh:1}\\
 u^-(x) &= \sup \left\{ t \in [-\infty,\infty] : \lim _{r\rightarrow 0} \frac{\mathcal{L}^{d}(\{v<t\}\cap B(x,r))}{r^d}=0\right\},\label{mh:2}
\end{align}
are the approximate upper and lower limits of $u$, respectively.
With these definitions, the total variation of the measure $D^{j}u$ can be written as
\[|D^{j}u|(\om)=\int_{J_{u}} |u^{+}(x)-u^{-}(x)|\, d\mathcal{H}^{d-1},\]
where $\mathcal{H}^{d-1}$ denotes the $(d-1)$-dimensional Hausdorff measure.
The density functions of the measures $D^{j}u$ and $D^{c}u$ with respect to $|D^{j}u|$ and $|D^{c}u|$ are denoted by $\sigma_{D^{j}u}$ and $\sigma_{D^{c}u}$, respectively.
For further details about the space $\bv(\om)$, we refer the reader to \cite{AmbrosioBV, EvansGariepy}.\\


\subsection{The space $W^{q}(\di;\om)$}

As the Banach space $W^{q}(\di;\om)$, with $q \in [1,\infty)$,  plays a major role in our work, we recall here some basic facts. 

\begin{defn}[$ W^q(\di;\Omega) $]
\label{def:Wq(div)}Let $ 1\leq q <\infty  $ and $g\in L^{q}(\Omega,\mathbb{R}^{d})$. We have
$\di g\in L^{q}(\Omega)$ if there exists $w\in L^{q}(\Omega)$
such that for all $\phi\in C_{c}^{\infty}(\Omega)$\[
\int_{\Omega}\nabla \phi \cdot g\,dx=-\int_{\Omega}\phi w\,dx.\] 
Furthermore we define \begin{equation*}
W^{q}(\di;\Omega):=\left\{ g\in L^{q}(\Omega,\mathbb{R}^{d}) : \di g\in L^{q}(\Omega)\right\},
\end{equation*}
with the norm $\Vert g\Vert_{W^{q}(\di;\om)}^{q}:=\Vert g\Vert_{L^{q}(\om)}^{q}+\Vert\di g\Vert_{L^{q}(\om)}^{q}.$\end{defn}

\begin{rem}
By density of $C_{c}^{\infty}(\Omega)$ in $L^{p}(\Omega)$ we have $\di g=w$ as $w\in L^{q}(\Omega)$ is unique. By completeness of $L^{q}(\Omega)$ and $L^{q}(\Omega,\mathbb{R}^{d})$ it follows that $W^{q}(\di;\om)$ is a Banach space when equipped with $\Vert\cdot\Vert_{W^{q}(\di;\om)}$.
\end{rem}

We now state some general properties of $W^{q}(\di;\om)$. As $ W^q (\di;\om) $ is just a straightforward generalization of the well-known space $ H(\di;\om ):= W^2 (\di;\om)$, these results can be proven readily by generalizing from $ H(\di;\om) $; see \cite[Chapter 1]{Girault} for details on the latter.

\begin{prop}[Density] \label{pro:Wq_div_(Density)} 
Let $q \in [1,\infty)$. Then
$C^{\infty}(\overline{\Omega},\mathbb{R}^{d})$
is dense in $W^{q}(\di;\Omega)$ with respect to $\Vert\cdot\Vert_{W^{q}(\di; \om)}$.
\end{prop}

\begin{prop}[Normal trace]
\label{pro:W(div)-trace} Let $q \in [1,\infty)$ and denote by $n_{\Omega}(x)\in\mathbb{R}^{d}$ the outer normal vector
to $\partial\Omega$ at $x\in\partial\Omega$. Then the mapping 
$$
\tau:C^{\infty}(\overline{\Omega},\mathbb{R}^{d})  \rightarrow  \left(W^{1-\frac{1}{p},p}(\partial\Omega)\right)^{*},\qquad
g  \mapsto  \tau(g),
$$
with $\tau(g)(v):=\int_{\partial\Omega}(g,n_{\Omega})v\, d\mathcal{H}^{d-1}$ for
$v\in W^{1-\frac{1}{p},p}(\partial\Omega)$, can be extended to a
linear, continuous mapping, also denoted by $\tau:W^{q}(\di;\Omega)\rightarrow\left(W^{1-\frac{1}{p},p}(\partial\Omega)\right)^{*}$.
Further we have a Gauss-Green formula for $W^{q}(\di;\om)$ functions:\[
\int_{\Omega}\nabla v\cdot g\,dx+\int_{\Omega}v\,\di g\,dx=\tau(g)(v)\quad\text{for all }v\in W^{1,p}(\Omega),\;g\in W^{q}(\di;\Omega).\]
\end{prop}

\begin{defn}
\label{def:W0(div)}For $ 1\leq q < \infty $, we define \[
W_{0}^{q}(\di;\Omega)=\overline{C_{c}^{\infty}(\Omega,\mathbb{R}^d)}^{\Vert\cdot\Vert_{W^{q}(\di;\om)}}.\]
\end{defn}

The next proposition also uses ideas from \cite{Girault} but since its proof is slightly more involved, we include it in Appendix \ref{sec:app}.
\begin{prop}
\label{pro:equiv_W0_div} For $q \in [1,\infty)$ we have $W_{0}^{q}(\di;\Omega)=\left\{ g : \tau(g)\equiv0\right\} =:\ker(\tau)$,
with $\tau$ the $W^{q}(\di;\Omega)$-trace operator as in Proposition
\ref{pro:W(div)-trace} and $\tau(g)\equiv0$ is understood in the sense of $\left(W^{1-\frac{1}{p},p}(\partial\Omega)\right)^{*}$.\end{prop}

From this fact, another equivalent characterization of $W_{0}^{q}(\di;\Omega)$
can be obtained readily.
\begin{cor}
\label{cor:other_equv_W0_div}For $q \in [1,\infty)$ and $g\in W^{q}(\di;\Omega)$ we have
$g\in W_{0}^{q}(\di;\Omega)$ if and only if\begin{equation}
\int_{\Omega}\nabla v\cdot g\,dx=-\int_{\Omega}v\,\di g\,dx,\quad\text{for all }v\in W^{1,p}(\Omega).\label{eq:anoth_equv_W0dive}\end{equation}
\end{cor}

\section{Structural TV as a lower semi-continuous envelope}\label{sec:relaxation}

The goal of this section is to obtain a predual representation of a general TV-based functional that includes the case of weighted TV for a general choice of weights. To this aim, we will define the corresponding functional as the $L^{p}$-lower semi-continuous envelope of a restriction to $W^{1,1}$ functions, with $p\in [1,\infty)$. The approach is motivated by the paper of Bouchitte and Dal Maso \cite{DalMaso93}. We start with a few definitions.

By $j:\Omega \times \RR^d \rightarrow [0,\infty)$ we always denote a function satisfying the following conditions:
\begin{itemize}
\item[(J1)] For a.e. $x\in \Omega$, $j(x,\cdot)$ is convex and positively 1-homogeneous on $\RR^d$.
\item[(J2)] There exists  $\gamma>0$ such that
\[0\le j(x,z) \le \gamma (1+|z|)\quad\text{for a.e. }x\in \Omega\text{ and every }z\in \RR^d.\]
\item[(J3)] For every $z\in\RR^{d}$, $j(x,z)=j(x,-z)$, i.e., $j$ is an even function in the second variable.
\end{itemize}

Furthermore, for $p \in [1,\infty)$ we define $J:L^{p}(\Omega) \rightarrow	 [0,\infty]$  as
\begin{equation} \label{eq:original_reg_func}
J(u) := \begin{cases}
\int _\Omega j(x,\nabla u(x) ) \,dx &\text{if } u \in W^{1,1}(\Omega), \\
\infty & \text{else.}
\end{cases}
\end{equation}

\begin{rem}
We note that $j(x,z) = \alpha(x)|z|$ with any $\alpha :\Omega \rightarrow [0,\infty)$ bounded above satisfies the above assumptions (J1)--(J3).
\end{rem}

In what follows, convex conjugation of $j$ will always be understood with respect to the second argument, and convex conjugation of $J$ is performed in $L^{p} (\Omega)$. 
Due to positive 1-homogeneity, we get the following well known representation of $j^*$. For its formulation, we define the support function
\begin{equation}\label{mh:supp}
j^\circ(x,z^*):= \underset{z:j(x,z) \leq 1}{\sup}z^*\cdot z,
\end{equation}
and denote by $j^*$ the convex conjugate of $j$; see \cite{Ekeland} for more information on the latter concept.

\begin{prop}\label{lbl:j_star}
Let $z^*\in\mathbb{R}^d$. Then, for any $x \in \Omega$ we have 
$ j^*(x,z^*) = \mathcal{I}_{A_x} (z^*)$, 
 with $\mathcal{I}_{A_x}$ the convex indicator function of the set 
$A_x = \{ z \in \RR^d : j^\circ (x,z) \leq 1\}$.
\end{prop}

\begin{proof}
In case $j(x,\cdot) = 0$ the assertion holds true trivially. Hence we assume that there exists $z \in \RR^d$ with $j(x,z) \neq 0$. 
For such a point $z$, we can write
\[ j(x,z) =  j\left (x,j(x,z)\frac{z}{j(x,z)}\right) =  \lambda j(x,\tilde{z}) \]
with $\lambda = j(x,z) \geq 0$ and for $\tilde{z}$ such that $j(x,\tilde{z}) =1$. 
Hence we get
\[ j^ \circ (x,z^ *) = \sup_{\lambda \in [0,1]} \sup_{j(x,z) =1} \lambda (z^ *\cdot z) = \sup_{j(x,z) =1} z^ *\cdot z\geq 0, \] 
where the non-negativity follows from the fact that $j(x,\cdot)$ is even.
We further have 
\begin{align*}
j^ *(x,z^ *) 
&= \sup _z \,\{z^ *\cdot z - j(x,z)\} = \sup _{\lambda \geq 0} \sup _{j(x,z)=1} \{ z^*\cdot (\lambda z) - \lambda j(x, z) \} \\
&= \sup _{\lambda \geq 0} \sup _{j(x,z)=1} \lambda ( z^*\cdot z - 1 )= \sup _{\lambda \geq 0 } \lambda (j^ \circ(x,z^ *) - 1) \\
&= \left.\begin{cases} \infty &\text{if } j^ \circ (x,z^ *) > 1, \\ 0 & \text{if } j^ \circ (x,z^ *) \leq  1\end{cases}\right\} = \mathcal{I}_{A_x}(z^*),
\end{align*}
which completes the proof.

\end{proof}

\begin{rem}\label{lbl:j_star_alpha}
From the proposition above if follows for $j(x,z) = \alpha(x)|z|$ that  
\[
j^\circ (x,z^*) =
 \begin{cases} 
\frac{|z^*|}{\alpha(x)} &\text{if } \alpha(x) \neq 0, \\
\mathcal{I}_{\{0\}} (z^*) &\text{else,}
\end{cases}
\]
and hence
$ A_x = \{ z^* \in \RR^d :  |z^*| \leq \alpha (x) \}$.
\end{rem}

We note that by definition and density we have for $u^* \in L^q(\Omega)$ with $q = \frac{p}{p-1}$ that
\begin{align*}
J^*(u^*) & = \sup _{u \in L^{p}(\Omega)} \{(u^*,u) - J(u)\} = \sup _{u \in \bv(\Omega)} \{(u^*,u) - J(u)\} 
= \sup _{u \in W^{1,1}(\Omega)} \{(u^*,u) - J(u) \}.
\end{align*} 
Here we write $(v,w):=\int_\Omega v w\,dx$ for $v\in L^q(\Omega)$ and $w\in L^p(\Omega)$. 
The next proposition follows the lines of \cite{DalMaso93} and provides a characterization $J^*$, which also holds without the positive 1-homogeneity assumption on $j$. 

\begin{prop}\label{lbl:J_star}
Let $p\in (1,\infty)$. Then we have for all $u^* \in L^q(\Omega)$ that 
\[ J^*(u^*) = \min _{g \in K(u^*)} \int _\Omega j^*(x,g(x))\, dx ,\]
where $K(u^*):=\{ g \in W^q_0 (\di;\om)\cap L^\infty(\Omega,\RR^d) : -\di g = u^* \}$, and we set $\min \emptyset := +\infty$.
\end{prop}

\begin{proof}
We define the map $F:L^1(\Omega,\RR^d) \rightarrow \overline{\RR}$ with $F(p) := \int _\Omega j(x,p(x))\,  dx$. Due to (J2), $F$ is well defined. Then, by \cite[Theorem X.2.1]{Ekeland}, we get for $F^*:L^\infty(\Omega,\RR^d) \rightarrow \overline{\RR} $ that
\[ F^*(p^*) = \int _\Omega j^*(x,p^*(x))\, dx .\]
Further, we define the unbounded operator $\Lambda:L^{p}(\Omega) \rightarrow L^1(\Omega,\mathbb{R}^d)$ by $\mathrm{dom}(\Lambda) = W^{1,1}(\Omega)\cap L^p(\Omega)$ and $\Lambda u := \nabla u$. Then $\Lambda$ is densely defined and closed. We note that $J$ can be re-written as
\[ J(u) = \begin{cases} 
F(\Lambda u) &\text{if } u \in \mathrm{dom}(\Lambda), \\
+\infty &\text{else.}
\end{cases}
\] Due to $|F(p) |\leq \gamma \operatorname{meas}(\Omega) + \gamma \|p\|_{L^{1}(\om)}$ by (J2), where $\operatorname{meas}(\cdot)$ denotes the Lebesgue measure of a set, $J$ is bounded in a neighborhood of any $\Lambda u$ with $u\in \mathrm{dom}(\Lambda)$. Hence, it follows  from \cite[Theorem 19]{Rockafellar74} that
\begin{align*}
J^* (u^*) & = \min \left\{ F^*(y^*) : y ^*\in \mathrm{dom} (\Lambda^*), \, \Lambda^*y^* = u^* \right\}\\
& = \min \left\{ \int _\Omega j^*(x,y^*(x))\, dx : y ^*\in \mathrm{dom} (\Lambda^*), \, \Lambda^*y^* = u^* \right\}.
\end{align*} 
To complete the proof, it is left to show that we have for any $u^* \in L^q(\Omega)$ that
\[ \{ y ^*\in \mathrm{dom} (\Lambda^*) : \Lambda^*y^* = u^* \} = \{ g \in W^q_0 (\di;\om)\cap L^\infty(\Omega,\RR^d) : -\di g = u^* \}. \]
We first show the inclusion ``$\supset$'': In case the set on the right-hand side above is empty, then the inclusion holds trivially; otherwise take $g \in W^q_0 (\di;\om)\cap L^\infty(\Omega,\RR^d)$ with $-\di g = u^*$ and $v \in \mathrm{dom}(\Lambda) \subset W^{1,1}(\Omega)$. By density \cite[Theorem 4.3]{EvansGariepy} there exists a sequence $(v_n)_{n\in\NN}$ in $W^{1,1}(\Omega) \cap C^\infty (\overline{\Omega})$ converging to $v$ in $W^{1,1}(\om)$ for which we can also assume that $v_n \rightarrow v $ in $L^{p}(\Omega)$ since $v \in L^{p}(\Omega)$. Then, by the Gauss-Green theorem for $W_0^q(\di;\om)$, see  \eqref{eq:anoth_equv_W0dive}, we get
\[ \int _\Omega \di g \,v\,dx\leftarrow \int _\Omega \di g \,v_n\,dx = - \int _\Omega g \cdot \Lambda v_n \,dx\rightarrow - \int _\Omega g\cdot \Lambda v\,dx.\]
Hence $|(\Lambda v ,g )| \leq C_g \|v\|_{L^{p}(\om)}$ with some $C_g>0$, and thus $g \in \mathrm{dom} (\Lambda^*)$ and $-\di g = \Lambda ^* g$. To show the reverse inclusion ``$\subset$'', again assuming the set on the left-hand side to be non-empty, we take $y^* \in \mathrm{dom} (\Lambda^*)\subset L^\infty (\Omega,\RR^d)$, for which we want to show that $y^* \in W^q_0 (\di;\om)$ and $\Lambda^* y^* = -\di y^*$. By definition of $\mathrm{dom} (\Lambda^*)$ the mapping $v\mapsto \int _\Omega \Lambda^*y^* v\,dx$ is a continuous linear functional on $L^{p}(\Omega)$, hence there exists $w \in L^q(\Omega)$ such that 
\[ \int _\Omega w v\,dx = \int _\Omega \Lambda^*y^* v\,dx = \int _\Omega y^* \cdot \nabla v \,dx\]
for all $v \in \mathrm{dom} (\Lambda) \supset C^\infty (\overline{\Omega}) \supset C^\infty _c (\Omega)$. Hence, $y^* \in W^q(\di;\om)$. Further, we also get that $y \in \ker(\tau)$, with $\tau$ being the normal trace operator of Proposition \ref{pro:W(div)-trace}, and hence $y^* \in W^q_0 (\di;\om)$, which completes the proof.
\end{proof}
Using positive 1-homogeneity, Proposition \ref{lbl:j_star} immediately implies the following refinement.
\begin{cor}\label{lbl:J_star_alpha}
Under the assumption of Proposition \ref{lbl:J_star} we get that
\[ J^*(u^*) = \mathcal{I}_{\di (Q)}(u^*) \]
with
\[ Q = \left\{ g \in W^q_0(\di;\om) \cap L^\infty(\Omega,\RR^d) :\, j^\circ(x,g(x)) \leq 1 \text{ for a.e. } x\in \Omega \right\} .\]
In particular, if $j(x,z) = \alpha(x)|z|$ then we get
\[ Q = \left\{  g \in W^q_0(\di;\om) \cap L^\infty(\Omega,\RR^d) :\, |g(x)| \leq \alpha(x) \text{ for a.e. } x\in \Omega \right\} .\]
\end{cor}


\subsection{Explicit representation of $J^ {**}$}
Now we study $J^{**}$ and aim at providing an explicit representation. Note that since $J$ is convex, $J^{\ast\ast}$ is equal to
the lower semi-continuous envelope of $J$ with respect to both the strong and weak $L^{p}$-convergence. In other words, for every $u\in L^{p}(\om)$ we have
\begin{align}
J^{\ast\ast}(u)
&=\inf \left\{ \liminf_{n\rightarrow\infty} J(u_n) : u_n \in L^{p}(\Omega), \, u_n \to u \text{ in } L^{p}(\Omega) \right\}\nonumber\\
&=\inf \left\{ \liminf_{n\rightarrow\infty} J(u_n) : u_n \in L^{p}(\Omega), \, u_n \rightharpoonup u \text{ in } L^{p}(\Omega) \right\},\label{Jastast_lsc_Ldast}
\end{align}
where ``$\to$'' refers to convergence with respect to the strong topology and ``$\rightharpoonup$'' with respect to the weak topology.

In order to obtain an explicit representation of $J^{**}$, we will employ the representation of the $L^{1}$-lower semi-continuous envelope of $J$, denoted here by $\overline{J}$, as derived in \cite{Fusco08}. Note that for $u\in L^{p}(\om)$
 \begin{align}\label{lsc_L1}
 \overline{J}(u) 
 & :=\inf \left\{ \liminf_{n\rightarrow\infty} J(u_n) : u_n \in W^{1,1}(\Omega), \, u_n \rightarrow u \text{ in } L^1(\Omega) \right\} \\
 &=\inf \left\{ \liminf_{n\rightarrow\infty} J(u_n) : u_n \in W^{1,1}(\Omega), \, u_n \rightharpoonup u \text{ in } L^1(\Omega) \right\},
 \end{align}
 with the last equality again being true due to convexity of $J$. 

We  show that, under suitable coercivity assumptions on $j$, the functionals $J^{\ast\ast}$ and $\overline{J}$ coincide.

\begin{lem}\label{lbl:J1_J2}
Assume $p\in [1,\ds]$ and that for some $c>0$ we have $c|z|\le j(x,z)$ for every $z\in\RR^{d}$ and for almost every $x\in\om$.
Then, $J^{\ast\ast}(u)=\overline{J}(u)$ for all $u\in L^{p}(\Omega)$.
\end{lem}

\begin{proof}
Since we assume $\Omega$ to be bounded, $L^ p$ convergence implies $L^ 1$ convergence. Also, $J(u_n) = +\infty$ for $u \in L^p(\Omega) \setminus W^{1,1}(\Omega)$ and, consequently, $\overline{J} \leq J^{**}$. Hence we are left to show $J^{**}(u) \leq \overline{J}(u)$ for all $u \in L^p (\Omega)$. To this aim, we first note that, due to the coercivity assumption on $j$, $c|Du|(\om)\le J(u)$ for all $u\in\bv(\om)$.

Now take  $a=\lim_{n\to\infty} J(u_{n})$ with $u_{n}\in W^{1,1}(\om)\subset \bv(\om)$, $u_{n}\to u$ in $L^{1}(\om)$. Without loss of generality, we can assume $a<\infty$. From the continuous embedding of $\bv(\om)$ into $L^{d^*}(\om)$, we get for a generic constant $C>0$ and for some $K>0$
\[\|u_{n}\|_{L^{p}(\om)}\le C(\|u_{n}\|_{L^{1}(\om)}+|Du_{n}|(\om))\le C(\|u_{n}\|_{L^{1}(\om)}+J(u_{n}))<K<\infty.\]  Hence, $(u_{n})_{n\in\NN}$ also converges weakly (up to subsequences) in $L^{p}(\om)$, and thus also weakly in $L^{1}(\om)$. By uniqueness of the weak limit we get $a = \liminf _{i\rightarrow \infty} J(u_{n_i})$ with $u_{n_i} \rightharpoonup u$ in $L^p(\Omega)$ and the proof is complete.
\end{proof}
\begin{rem}\label{u_BV_finite_J}
Note that if the  coercivity assumption of Lemma \ref{lbl:J1_J2} holds, then due to the lower semi-continuity of total variation with respect to weak $L^{1}$-convergence we get $u\notin \bv(\Omega)$ if and only if $\overline{J}(u)=J^{**}(u)=+\infty$.
\end{rem}

We then get the following result, which is a direct consequence of \cite[Theorem 1.1]{Fusco08}.

\begin{prop}[Integral representation of $J^{\ast\ast}$]\label{lbl:relaxation}
 Assume that $p \in [1,\ds]$ and one of the following two assertions holds true:
 
 \begin{enumerate} \setlength\itemsep{0.5em}
 \item $j(x,z) = \alpha(x)b(z)$ with $\alpha\in \BV(\Omega)$ and $b$ being a convex function such that (J1)--(J3) hold for $j$.
 \item There exists a constant $c>0$ such that for every $z\in\RR^{d}$ and for almost every $x\in\om$, $c|z|\leq j(x,z)$ and also $j(\cdot,z)\in \BV(\Omega)$.
 \end{enumerate}
Then, for any $u\in\bv(\om)$ we have
\begin{equation}\label{eq:lsc_relaxtion_of_j}
J^{**} (u) = \int _\Omega j(x,\nabla u(x))\,dx + \int _\Omega j^-(x,\sigma_{D^cu}) \,d |D^cu| + \int _{J_u \cap \Omega} ((u^+(x) - u^-(x) )j^- (x,\sigma _{D^ju})\,d \mathcal{H}^{d-1} .
\end{equation}
\end{prop}

\begin{proof}
Denote by $\mathcal{J}$  the right-hand side of equation \eqref{eq:lsc_relaxtion_of_j}.
If one of the two assumptions is satisfied, then it follows from \cite[Theorem 1.1]{Fusco08} that $\overline{J}(u)=\mathcal{J}(u)$ for every $u\in \BV(\Omega)$. 
It hence remains to show that $J^{**}(u) = \overline{J}(u)$ for any $u\in \BV(\Omega)$. 
In case $(ii)$ is satisfied, this is the assertion of Lemma \ref{lbl:J1_J2}. Assume now that $(i)$ holds true. In that case we will show directly that $J^{\ast\ast}(u)=\mathcal{J}(u)$ for every $u\in\BV(\Omega)$. It follows from \cite[Theorem 3.1]{Fusco08} that $\mathcal{J}$ is lower semi-continuous with respect to $L^1$ convergence. Consequently it is also lower semi-continuous with respect to weak $L^p$ convergence and hence, for all $u$ and $(u_n)_{n\in\NN}$ in $ W^{1,1}(\Omega)$ with $u_n \rightharpoonup u$ in $L^p$ we get
\[\mathcal{J}(u) \leq \liminf_{n\to\infty} J(u_n) .\]
This means that $\mathcal{J}(u)$ is a lower bound for the set 
\[ 
\left\{ \liminf_{n\rightarrow\infty} J(u_n) : u_n \in W^{1,1}(\Omega), \, u_n \rightharpoonup u \text{ in } L^{p}(\Omega) \right\}
\]
and since, by definition of the lower semi-continuous relaxation, $J^{**}(u)$ is the infimum of this set, we get
\[ \mathcal{J}(u) \leq J^{**}(u).\]
Furthermore, in the proof of \cite[Theorem 4.1]{Fusco08} it is shown, with $u_h = u\ast \phi_h$ being a standard mollification of $u\in \BV(\Omega)$ and any $A\subset \subset \Omega$ with $|\Du|(\partial A) = 0$, that
\[ \liminf_{h\rightarrow \infty}J(u_h,A) \leq \mathcal{J}(u), \]
	where, for any $A \subset \Omega$ open, $J(u,A): = \int _A j(x,\nabla u(x))\,dx $ if $u \in W^{1,1}(\Omega)$ and $J(u,A) = +\infty $ otherwise. Denoting, for fixed $A \subset \Omega$, $\overline{J}^{\ds}(\cdot,A)$ the lower semi-continuous  relaxation of $J(\cdot,A)$ in $L^\ds$, we get from convergence of $u_h$ to $u$ in $L^\ds(\Omega)$ as $h\to+\infty$ (which holds, since by embedding $u \in L^\ds(\Omega)$) that
\[ \overline{J}^{\ds}(u,A) \leq \liminf_{h\rightarrow \infty}J(u_h,A) \leq \mathcal{J}(u). \]
Now by \cite[Theorem 4.1]{DalMaso93}, the function $A \mapsto \overline{J}^{\ds}(u,A)$ can be extended to  a non-negative Borel measure. We define the uncountable family of open sets $(\Omega_\epsilon)_{\epsilon > 0}$ by $\Omega_\epsilon := \{x \in \Omega : d(x,\partial\Omega) > \epsilon\}$. Then $\Omega_\epsilon \subset \subset \Omega $ and $\partial\Omega_\epsilon \cap \partial\Omega_{\epsilon'} = \emptyset$. By sigma additivity  and since $\overline{J}^{\ds}(u,\Omega) < +\infty$ we get that, for any $n \in \N$, the set
\[ \{ \epsilon > 0 : \overline{J}^{\ds}(u,\partial\Omega_\epsilon) > 1/n \} \]
is finite. Hence, the set $\{ \epsilon > 0 : \overline{J}^{\ds}(u,\partial\Omega_\epsilon) > 0 \}$ is at most countable and we can extract a sequence $(\Omega_{\epsilon_i})_{i\in\NN}$ such that $\overline{J}^{\ds}(u,\partial\Omega_{\epsilon_i}) = 0$, 
$(\epsilon_i)_{i\in\NN}$ is monotonically decreasing
and $\bigcup _i \Omega _{\epsilon_i} = \Omega$. Hence we conclude (by $p \leq \ds$ for the left-most inequality below)
\[ J^{**}(u) \leq \overline{J}^{\ds}(u, \Omega ) = \lim _{i\rightarrow \infty} \overline{J}^{\ds}(u,\Omega _{\epsilon_i}) \leq \mathcal{J}(u) \]
and the proof is complete.
\end{proof}

%
%
%

\begin{rem}\label{lbl:remark_relaxation}
 In the case $j(x,z) = \alpha(x)|z| $ with $0\leq \alpha(x) \leq C$ and $\alpha \in \bv(\Omega)$, the assumptions of the above proposition are fulfilled. Using that, in this case $(j^-) = \alpha^- |\cdot |$, we get that 
\begin{align*}
J^{**}(u) &= \int _\Omega \alpha^- |\nabla u(x)|\,dx + \int _\Omega \alpha ^- |\sigma_{D^cu}| \,d |D^cu| + \int _{J_u \cap \Omega} ((u^+ - u^- )\alpha^- |\sigma _{D^ju}|\,d \mathcal{H}^{d-1} \\
&=  \int_{\om} \alpha ^- \,d |Du|.
\end{align*}

\end{rem}

We provide further remarks concerning the case $j(x,z)=\alpha(x)|z|$ in Section \ref{sec:refine_weightedTV} below.

\subsection{Subdifferential characterization}
From Proposition \ref{lbl:J_star} and Corollary \ref{lbl:J_star_alpha} we directly obtain an integral characterization of the subdifferential for a wide class of structural total variation type functionals. The corresponding result is established in Proposition \ref{lbl:subdiff_Jstarstar} below. Also, we refer to \cite{Anzellotti86, Holler12_subdif,Chambolle_Goldman_12} for more elaborate results in that direction.

\begin{prop}\label{lbl:subdiff_Jstarstar}
Let $p \in (1,\infty)$. Then, for $u \in L^ p(\Omega)$ we have  $u^{\ast}\in \partial J^{\ast\ast}(u)\subset L^{q}(\om)$ if and only if 
\begin{enumerate}
\item $J^{\ast\ast}(u)<+\infty$, and
\item there exists a vector field $g\in W_{0}^{q}(\di;\om)\cap L^{\infty}(\om,\RR^{d})$ such that
\[u^{\ast}=\di g,\qquad  J^{\ast\ast}(u)=\int_{\om} u\, \di g \,dx,\qquad j^{\circ}(x,g(x))\le 1\quad \text{ for a.e. }x\in\om.\]
\end{enumerate}
\end{prop}

\begin{proof}
Let $u\in L^{p}(\Omega)$. We have (see, e.g., \cite{Ekeland})
\begin{align*}
u^{\ast} \in \partial J^{\ast\ast}(u) 
&\iff J^{\ast\ast}(u)+J^{\ast\ast\ast}(u^{\ast})=\int_{\om} u^{\ast}u\,dx\\
&\iff J^{\ast\ast}(u)+J^{\ast}(u^{\ast})=\int_{\om} u^{\ast}u\,dx\\
&\iff J^{\ast\ast}(u)+ \mathcal{I}_{\mathrm{div}(Q)}(u^{\ast})=\int_{\om} u^{\ast}u\,dx,
\end{align*}
where $Q$ is defined in Corollary \ref{lbl:J_star_alpha} as
\[ Q = \left\{ g \in W^q_0(\di;\om) \cap L^\infty(\Omega,\RR^d) :\, j^\circ(x,g(x)) \leq 1 \text{ for a.e. } x\in \Omega \right\} .\]
Hence the result follows immediately.
\end{proof}

Recall that if $j$ satisfies the coercivity assumption $c|z|\le j(x,z)$, then Remark \ref{u_BV_finite_J} implies that $J^{\ast\ast}(u)<+\infty$ is equivalent to $u\in\bv(\om)$.

\section{Some refinements for weighted TV}\label{sec:refine_weightedTV}

For the special case $j(x,z)=\alpha(x)|z|$ we next investigate alternative dual definitions of the structural total variation functional. This will provide some density results for pointwise bounded $W_0^d(\dive;\Omega)$-functions, which are of interest in a variety of applications \cite{sing_mol, Hint_Rau_density, hint_rau, hint_rau_ros}.

Given $u\in L^p(\om)$ with $p \in (1,\ds] $, consider the following two extended-valued weighted $\tv$-functionals:
\begin{align}
\tv_{\alpha}^{W}(u)&:= \sup\left \{ \int_{\om}u\, \di g\,dx : g\in W_{0}^{q}(\di;\om),\; |g(x)|\le \alpha(x),\text{ for a.e. }x\in\om \right \},\label{wTV_g}\\
\tv_{\alpha}^{C}(u)&:= \sup\left \{ \int_{\om}u\, \di \phi\,dx : \phi\in C_{c}^{\infty}(\om,\RR^{d}),\; |\phi(x)|\le \alpha(x),\; \forall x\in\om \right \},\label{wTV_phi}
\end{align}
Recall that $\tv_\alpha^W  = J^{**}$ by Corollary \ref{lbl:J_star_alpha}. 


We commence by stating a preparatory result whose proof can be found in \cite{sing_mol}.
\begin{lem}\label{lbl:smooth_below}
Let $\phi\in C_{c}(\Omega,\RR^{d})$. Then for every $\epsilon>0$ there exists a function $\phi_{\epsilon}\in C_{c}^{\infty}(\Omega, \RR^{d})$ such that
\[\|\phi-\phi_{\epsilon}\|_{\infty}<\epsilon\quad \text{and}\quad |\phi_{\epsilon}(x)|\le |\phi(x)|\quad \text{for all }x\in\Omega.\]
\end{lem}
Using this density result, the following result can be shown for a continuous weight function $\alpha$.
\begin{prop}\label{lbl:dual_weighted_TV}
 Let $\alpha\in C(\overline{\om})$ with $\alpha\ge 0$ and $u\in\bv(\om)$. Then
 \begin{equation}\label{dual_weighted_TV}
 \tv_{\alpha}^{C}(u)=\int_{\om}\alpha\,d|Du|.
 \end{equation}
\end{prop}

\begin{proof}
From \cite[Prop. 1.47]{AmbrosioBV} we have that
\begin{align}
\int_{\om}\alpha\,d|Du|	&=\sup \left \{\sum_{i=1}^{d} \int_{\om} \phi_{i}\alpha\,dD_{i}u : \phi\in C_{c}(\om,\RR^{d}),\;|\phi(x)|\le 1 ,\;\forall x\in \om\right \}\notag\\
								      &= \sup \left \{\sum_{i=1}^{d} \int_{\om} \phi_{i}\,dD_{i}u : \phi\in C_{c}(\om,\RR^{d}),\;|\phi(x)|\le \alpha(x),\;\forall x\in \om \right \}\label{prop147_a}\\
								      &\le \sup \left \{\int_{\om}|\phi\,|d|Du| :  \phi\in C_{c}(\om,\RR^{d}),\;|\phi(x)|\le \alpha(x),\;\forall x\in \om\right \} \notag\\
								      &\le \int_{\om}\alpha\,d|Du|.\notag
\end{align}
The proof will be completed by showing that \eqref{prop147_a} is equal to
\begin{equation}\label{prop147_a_diff}
 \sup \left \{\sum_{i=1}^{d} \int_{\om} \phi_{i}\,dD_{i}u : \phi\in C_{c}^{\infty}(\om,\RR^{d}),\;|\phi(x)|\le \alpha(x),\;\forall x\in \om \right \}.
\end{equation}
Indeed, if $\phi\in C_{c}^{\infty}(\om,\RR^{d})$, then $\sum_{i=1}^{d} \int_{\om} \phi_{i}\,dD_{i}u=-\int_{\om}u\,\mathrm{div}\phi\,dx$ and thus  \eqref{prop147_a_diff} is equal to the right-hand side of \eqref{dual_weighted_TV}. Lemma \ref{lbl:smooth_below} now allows to approximate any term in \eqref{prop147_a} by a term in \eqref{prop147_a_diff} satisfying the pointwise constraint,
which completes the proof.
\end{proof}


As we show next, \eqref{wTV_g} and \eqref{wTV_phi} are equivalent for $u\in\bv(\om)$ and a continuous weight function. 
\begin{prop}\label{lbl:TVphi_TVg_BV}
Let $\alpha\in C(\overline{\om})$ with $\alpha\ge 0$ and $u\in\bv(\om)$. Then
\begin{equation}\label{TVphi_TVg_BV}
\tv_{\alpha}^{C}(u)=\tv_{\alpha}^{W}(u)=\int_{\om} \alpha\,d|Du|.
\end{equation}
\end{prop}
\begin{proof}
We have that for every $u\in\bv(\om)$ there exists a sequence $(u_{n})_{n\in\NN}\subset C^{\infty}(\om)\cap W^{1,1}(\om)\cap L^{d^{\ast}}(\om)$ that converges to $u$ in $L^{d^{\ast}}(\om)$ and also $\int_{\om}\alpha\, d|Du_{n}|\to \int_{\om}\alpha\, d|Du|$; see for instance \cite[Thm. 3.9 \& Prop. 3.15]{AmbrosioBV}. This, together with the fact that $\int_{\om}\alpha\,d|D\cdot|$ is lower semi-continuous with respect to $L^{d^{\ast}}$-convergence, and also that $J^{\ast\ast}$ is the $L^{d^{\ast}}$-lower semi-continuous envelope of $J$ implies that
\[\tv_{\alpha}^{C}(u) = \int_{\om}\alpha\,d|Du|=J^{\ast\ast}(u) = \tv_{\alpha}^{W}(u),\]
where the first equality is due to Proposition \ref{lbl:dual_weighted_TV}  and the last equality is due to Corollary \ref{lbl:J_star_alpha}.
\end{proof}

Note that in the proof above, we are not able to  directly use the result of \cite{Fusco08} or its corollaries proven in the previous section, as $\alpha\in C(\overline{\om})$ does not imply $\alpha\in \bv(\om)$. For smooth (but not necessarily with integrable gradient) $u$ the following analogous result holds true; see Appendix \ref{sec:app} for its proof.

\begin{prop}\label{lbl:Tva_smooth}
Let  $\alpha\in C(\overline{\om})$ with $\alpha\ge 0$ and $u\in C^{1}(\om)$. Then
$\tv_{\alpha}^{C}(u)=\int_{\om} \alpha |\nabla u|\,dx.$
\end{prop}

We now reduce the regularity of $\alpha$ by assuming the following property for the function $\alpha^{-}\ge 0$:
\begin{equation}\label{P_lsc}
\tag{$P_{lsc}$} \text{There exists } (\alpha_{n})_{n\in\NN}\subset C(\overline{\om}),\;0\le \alpha_{n}(x) \le \alpha^{-}(x) \text{ such that } \alpha_{n}(x)\to \alpha^{-}(x), \quad \forall x\in\om.
\end{equation}
Observe that the above requirement slightly generalizes lower semi-continuity.

\begin{prop}\label{lbl:alpha_lsc}
Suppose that $\alpha\in \bv(\om)\cap L^{\infty}(\om)$, with $\alpha\ge 0$, and $\alpha^{-}$satisfies \eqref{P_lsc}. Then for all $u\in \BV(\Omega)$ we have
\begin{equation}\label{TVphi_TVg_alsc}
\tv_{\alpha^{-}}^{C}(u)=\tv_{\alpha}^{W}(u)=\int_{\om} \alpha^{-}\,d|Du|.
\end{equation}
\end{prop}
\begin{proof}
Since $\alpha^{-}$ satisfies \eqref{P_lsc}, we have $\alpha^{-}(x)=\sup \left\{ \tilde{\alpha}(x) : \tilde{\alpha}\in C(\om),\;\tilde{\alpha}\le \alpha\right\}$, which yields
\begin{equation}\label{sup1}
\sup_{\substack{\tilde{\alpha}\le \alpha^{-}\\ \tilde{\alpha}\in C(\overline{\om})}}\int_{\om} \tilde{\alpha}\, d|Du|
=\int_{\om}\alpha^{-}\,d|Du|.
\end{equation}
Moreover, we find
\begin{align*}
\sup_{\substack{\tilde{\alpha}\le \alpha^{-}\\ \tilde{\alpha}\in C(\overline{\om})}}\int_{\om} \tilde{\alpha}\, d|Du|&= 
 \sup_{\substack{\tilde{\alpha}\le \alpha^{-}\\ \tilde{\alpha}\in C(\overline{\om})}}  \sup \left \{\int_{\om}u\, \di\psi\,dx :  \psi\in C_{c}^{\infty}(\om,\RR^{d}),\; |\psi(x)|\le \tilde{\alpha}(x),\;\forall x\in\om \right \}\\
 &\le \sup \left \{\int_{\om}u\, \di\phi\,dx : \phi\in C_{c}^{\infty}(\om,\RR^{d}),\; |\phi(x)|\le \alpha^{-}(x),\;\forall x\in\om \right \}\\
 &=\tv_{\alpha^{-}}^{C}(u).
\end{align*}
But, the above inequality is, in fact, an equality as for every element $\int_{\om} u\,\di\phi \,dx$ with $\phi\in C_{c}^{\infty}(\om,\RR^{d})$ and $|\phi|\le \alpha^{-} $, we have
\[\int_{\om} u\,\di\phi \,dx\le \sup \left \{\int_{\om}u\, \di\psi\,dx : \psi\in C_{c}^{\infty}(\om,\RR^{d}),\; |\psi(x)|\le |\phi(x)|,\;\forall x\in\om \right \}.\]
From this, \eqref{sup1}, Corollary \ref{lbl:J_star_alpha} and Remark \ref{lbl:remark_relaxation}, we get that
\[\tv_{\alpha-}^{C}(u)=\tv_{\alpha}^{W}(u)=\int_{\om}\alpha^{-}\,d|Du|,\]
which ends the proof.
\end{proof}
Under a uniform positivity assumption on the weight, we obtain a density result, which is of use, e.g., when predualizing the renowned Rudin-Osher-Fatemi model; see \cite{Rudin92}, and \cite{Hintermueller04, hint_rau, hint_rau_ros} for its dualization.

\begin{prop}\label{lbl:TVphi_TVg_Ld_lsc}
Let the assumptions of Proposition \ref{lbl:alpha_lsc} hold true, and assume in addition that $\alpha>c>0$ a.e. in $\om$. Then we have for $q\in [d,+\infty)$ that
\begin{equation*}
\overline{\left\{ \di\phi : \phi\in C_{c}^{\infty}(\om,\RR^{d}),\; |\phi|\le \alpha^{-}   \right \}}^{L^{q}(\om)}= \left\{\di g : g\in W_{0}^{q}(\di;\om),\; |g|\le \alpha^{-},\;a.e. \right \}.
\end{equation*}
\end{prop}

\begin{proof}
For every $u\in\bv(\om)$ we have
\[
\tv_{\alpha^{-}}^{C}(u)= \int_{\om} \alpha^{-}d|Du| =\overline{J}(u)=J^{\ast\ast}(u)=\tv_{\alpha}^{W}(u)	,
\]
where the first equality stems from Proposition \ref{lbl:alpha_lsc}, the second one is due to \cite{Fusco08}, the third equality comes from Proposition \eqref{lbl:J1_J2}, and Corollary \ref{lbl:J_star_alpha} yields the final relation.

As $\alpha>c>0$ a.e. in $\om$, we have $\tv_{\alpha^{-}}^{C}(u)=\tv_{\alpha}^{W}(u)=+\infty$ for every $u\in L^{p}(\om)\setminus \bv(\om)$. Hence, we have $\tv_{\alpha^{-}}^{C}=\tv_{\alpha}^{W}$ in all of $L^ p(\Omega)$ which is equivalent to \begin{equation*}
\mathcal{I}_{\left\{ \di\phi : \phi\in C_{c}^{\infty}(\om,\RR^{d}),\; |\phi|\le \alpha^{-}   \right \}}^{\ast}(u)
= \mathcal{I}_{\left\{\di g : g\in W_{0}^{q}(\di;\om),\; |g|\le \alpha^{-},\;a.e. \right \}}^{\ast}(u),\quad \text{for all }u\in L^{p}(\om).
\end{equation*}
After dualization and using the fact that the second set in the equation above is closed in $L^{q}(\om)$ (compare \cite{Holler12_subdif} for a proof for scalar $\alpha$ which readily carries over to the present setting), we obtain for all $u^{\ast}\in L^{q}(\om)$ that
\begin{equation*}
\mathcal{I}_{\overline{\left\{ \di\phi : \phi\in C_{c}^{\infty}(\om,\RR^{d}),\; |\phi|\le \alpha^{-}  \right \}}^{L^{q}(\om)}}(u^{\ast})
= \mathcal{I}_{\left\{\di g : g\in W_{0}^{q}(\di;\om),\; |g|\le \alpha^{-},\;a.e. \right \}}(u^{\ast}),
\end{equation*}
which proves the claimed density.
\end{proof}

Density results of the above type have recently gained attention in the literature \cite{sing_mol, Hint_Rau_density, hint_rau,  hint_rau_ros} and enjoy a variety of applications. In the context of variational regularization in image reconstruction, an analogous density result for continuous weights $\alpha$ was used in \cite{hint_rau} in order to show equivalence of a weighted $\tv$-regularization problem and a corresponding predual problem; see \cite{Hintermueller04} for a scalar weight. We emphasize that the result of Proposition \ref{lbl:TVphi_TVg_Ld_lsc} allows for the dualization for a larger class of weights rather than continuous functions.  In the following section, we discuss an even more general duality result.

\section{A general duality result}\label{sec:duality}
In this section we consider the variational regularization of linear inverse problems using structural total variation regularization. Our goal is to show existence of a solution as well as equivalence to a saddle-point formulation in the continuous setting under mild conditions that are naturally satisfied by the applications of our interest. The saddle-point problem will be formulated in a way that it only requires an explicit form of $J^*$, but not of $J^{**}$, and such that its numerical solution by duality-based optimization algorithms is direct.

Since we aim to capture diverse applications, such as structural-TV-regularized MR and PET reconstruction, our main duality result will be rather general with technical assumptions. In order to better demonstrate the essence of our result, we first consider the particular case where data fidelity will be guaranteed by a norm discrepancy.
For this purpose, consider
\begin{equation} \label{eq:general_min_problem_norm}
\inf_{u\in L^p (\Omega)} J^{**}(u)  + \lambda\|Ku-f\|_{S},
\end{equation} 
where $p \in (1,\ds]$, $(S,\|\cdot \|_S)$ is a Banach space with $f \in S$, $K:L^p(\Omega ) \rightarrow S $ is a bounded linear operator (i.e., $K\in \mathcal{L}(L^p(\Omega ),S) $), $J^{**}$ corresponds to the structural TV functional as defined above, and $\lambda>0$ is a regularization parameter.
 
Note that, without further assumptions (see Propositions \ref{lbl:J_star} and \ref{lbl:relaxation}), we only have $J^* = J^ {***}$, but $J^{**}$ is not available explicitly. Hence we are interested in showing equivalence of \eqref{eq:general_min_problem_norm} to an appropriate predual problem which only requires $J^ *$. We will see that, for general $J$, this is possible if either $c|z| \leq j(x,z)$ for every $z\in\RR^{d}$ and for almost every $x\in\om$, or the inversion of $K$ is essentially well-posed, i.e., $K$ has a closed range and a finite dimensional kernel. Regarding the latter, we note that this is in particular true if we assume $K^ *K$ to be invertible, with $K^{\ast}$ being the adjoint of $K$.  A first result for the particular setting of \eqref{eq:general_min_problem_norm} is stated next.

\begin{prop}\label{prop:duality_norm_discrepancy}
Let $p \in (1,\ds]$, $(S,\|\cdot \|_S)$ a Banach space with $f \in S$, $K \in \mathcal{L}(L^p(\Omega ),S) $
 and $\lambda > 0$.
Assume that at least one of the following two conditions holds:
\begin{enumerate}
\item There exists $c>0$ such that $c|z| \leq j(x,z)$ for every $z\in\RR^{d}$ and for almost every $x\in\om$.
\item $\range(K^*)$ is closed and $\ker(K)$ is finite dimensional.
\end{enumerate}
Then there exists a solution to the primal problem 
\begin{equation}\label{eq:primal_problem_norm}
 \inf_{u\in L^p (\Omega)} J^{**}(u)  + \lambda\|Ku-f\|_{S}, 
 \end{equation}
to a corresponding  predual problem, and to the saddle-point problem 
\begin{equation} \label{eq:sp_problem_norm}
\inf_{\substack{p \in W^q_0(\dive;\om) \\ p \in Q}} \sup_{u\in L^p (\Omega)} (\di p ,u) - \lambda\|Ku-f\|_{S}, 
\end{equation}
where $Q = \{g \in W^q_0(\dive;\om)\cap L^\infty(\Omega,\R^d) : \;j^\circ(x,g(x))\leq 1 \text{ for a.e. } x \in \Omega\}$. 
They all coincide and are equivalent in the sense that the pair $(p,u)$ is a solution to the saddle-point problem if and only if $u$ is a solution to \eqref{eq:primal_problem_norm} and $p$ is a solution to the predual problem.
\begin{proof}
This is a special case of Theorem \ref{prop:general_duality_result} below. We, hence, refer to the corresponding proof.
\end{proof}
\end{prop}

We mention that the above result readily carries over to data fidelities of the type $\|Ku-f\|_S^r$, with $r\in [1,\infty)$. Here, often $r=2$ is of interest when $S$ is a Hilbert space, or $r=p$ when $S=W^{k,p}(\om)$, with $k\in\mathbb{N}_0$ and $p\in [1,\infty)$.

Note that the coercivity assumption $(i)$ of Proposition \ref{prop:duality_norm_discrepancy} excludes the case where $J^{\ast\ast}(u)=\int_{\om} \alpha \,d |\Du|$ with vanishing weight $\alpha$. The following example shows that if the weight is not bounded uniformly away from zero, existence for \eqref{eq:primal_problem_norm} is not guaranteed, in general (not in $\bv(\om)$ but also not even in $L^{2}(\om)$). This observation justifies assumption $(i)$ very well.

\begin{prop}\label{counterexample}
There exists $\om\subset \RR^{d}$, $\alpha\in C(\overline{\om})$ with $\alpha$ vanishing only at one point, a Banach space $S$, data $f\in S$ and an injective, bounded linear operator $K:L^{2}(\om)\to S$ such that the minimization problem
\begin{equation}\label{min_counterex}
\inf_{u\in L^{2}(\om)} \int_{\om}\alpha\,d|Du|+\|Ku-f\|_{S},
\end{equation}
does not have a solution.
\end{prop}

\begin{proof}
Let $d=2$ and define $\om=[-L,L]\times [-L,L]$. Moreover, let $\alpha$ be a positive continuous function on $\om$ that vanishes only at the origin and which satisfies
\begin{equation}
\alpha(x)\le \frac{1}{n^{2}}, \quad\text{for all }x\in M_{n}:=\left\{y\in\om : |y|=\frac{1}{n} \right\}.
\end{equation}
Define $S:=\left(\overline{C_{c}^{\infty}(\om)}^{\|\cdot\|}\right)^{\ast}$, where $\|\cdot\|:=\|\cdot\|_{\infty}+\|\nabla\cdot\|_{\infty}$, and $K:L^{2}(\om)\to S$ with
\[Ku(\phi)=\int_{\om}u\phi\,dx,\quad \text{for all }\phi \in \overline{C_{c}^{\infty}(\om)}^{\|\cdot\|}.\]
Note that $K$ is injective, linear and bounded. In fact, concerning the latter observe
\[\|Ku\|_{S}=\sup _{\|\phi\|\le 1} Ku(\phi)=\sup _{\|\phi\|\le 1} \int_{\om} u\phi\,dx \le \sup _{\|\phi\|\le 1} C\|\phi\|_{\infty}\|u\|_{L^{2}(\om)}\le \sup _{\|\phi\|\le 1} C\|\phi\|\|u\|_{L^{2}(\om)}\le C\|u\|_{L^{2}(\om)},\]
for an appropriate constant $C>0$. Finally choose $f=\delta_{0}$, i.e., the Dirac measure at zero. We claim that with this set-up the infimum in \eqref{min_counterex} is zero. Indeed, define $u_{n}:=\frac{1}{\operatorname{meas}(N_{n})}\mathcal{X}_{N_{n}}$ where $N_{n}:=\left\{x\in\mathbb{R}^2 : |x| \le \frac{1}{n}\right\}$. Given that $\operatorname{meas}(N_{n})=\frac{\pi}{n^{2}}$, we have 
\[\int_{\om}\alpha\,d|Du_{n}|\le \frac{1}{n^{2}} |Du_{n}|(M_{n}) \le \frac{1}{n^{2}}\frac{n^{2}}{\pi} \mathcal{H}(M_{n})= \frac{1}{n^{2}} \frac{n^{2}}{\pi} 2\pi \frac{1}{n}=\frac{2}{n}\to 0.\]
Now for the fidelity term we have 
\begin{align*}
\|Ku_{n}-\delta_{0}\|_{S}
&\le \sup_{\|\phi\|\le 1} \left | \int_{\om} u_{n} \phi\,dx - \phi(0) \right |\le  \sup_{\|\phi\|\le 1} \frac{1}{\operatorname{meas}(N_{n})}\int_{N_{n}} |\phi(x)-\phi(0)|\,dx\\
&\le  \sup_{\|\phi\|\le 1} \frac{\|\nabla \phi\|_{\infty}}{\operatorname{meas}(N_{n})}\int_{N_{n}} |x|_2\,dx \le \frac{1}{\operatorname{meas}(N_{n})} \int_{N_{n}} \frac{1}{n}\,dx \to 0.
\end{align*}
Hence, if there was a minimizer $\tilde{u}\in L^2(\om)$, then $K\tilde{u}$ should be equal to $\delta_{0}$ in the sense that $\int_{\om}\tilde{u}\phi\,dx=\phi(0)$ for every $\phi\in C_{c}^{\infty}(\om)$. But this is impossible and hence there is no minimizer for the problem \eqref{min_counterex}.

\end{proof}

 Observe that the operator $K$ in the proof of Proposition \ref{counterexample} has no closed range. This can be seen easily as for the sequence $(u_{n})_{n\in\NN}$, in the proof above, we have $\|Ku_{n}-\delta_{0}\|_{S} \to 0$ and $\delta_{0}\notin \mathrm{Rg}(K)$. From the closed range theorem we also get that the range of $K^{\ast}$ is not closed.
 As $K$ is injective, and, in particular, it then has finite dimensional kernel, all other assumptions of Proposition \ref{prop:duality_norm_discrepancy} are satisfied. Hence the closed range assumption is tight in the sense that, for non-closed range operators we cannot expect a similar result without further assumptions on the integrand $j$.

Motivated by the particular case of a norm discrepancy as data fidelity, we now consider a more general setting. 
In fact, let $p \in (1,\ds]$, $\lambda >0$, $K \in \mathcal{L}(L^ p(\Omega),S)$ with $S$ being a Banach space, and assume $D:S \rightarrow \overline{\RR}$ to be convex and lower semi-continuous. We aim to solve
\begin{equation} \label{eq:general_problem}
\inf_{u\in L^p (\Omega)} J^{**}(u)  +  (\lambda D \circ K)(u),
\end{equation} 
where $J^{**}$ again corresponds to the structural TV functional as defined above. We recall first the following result which follows from  \cite[Theorem 19]{Rockafellar74}.
\begin{lem} \label{lem:domain_predual_vectorspace} Let $S$ be a Banach space and $D:S \rightarrow \overline{\R}$ convex, lower semi-continuous, and continuous and finite at zero. Further, let $K:L^p(\Omega) \rightarrow S$ be a bounded linear operator. Then for all $x^* \in L^q(\Omega)$ we have
\[(D \circ K)^*(x^*) = \min_{\substack{s^* \in S^* \\ K^* s^* = x^*}} D^*(s^*),\]
where the minimum is attained. Consequently, it holds that $\domain((D \circ K)^*) = K^* \domain (D^*)$.
\end{lem}
We will also need the following generalization of an orthonormal decomposition of $L^q(\Omega)$ (see for instance \cite[Corollary 6.1]{Holler_ictv}): For $q \in (1,\infty) $, we denote the space of constant functions in $L^q(\Omega)$ by $\ker(\nabla)$. Then, with 
\[\ker(\nabla)^\perp := \{ u \in L^q(\Omega)  : (v,u) = 0  \text{ for all } v \in L^p(\Omega), \text{ with } v \text{ a constant function} \} ,\] 
we get that $\ker(\nabla)^\perp$ is closed, $L^q(\Omega) = \ker(\nabla) + \ker(\nabla)^\perp$ and $\ker(\nabla) \cap \ker(\nabla)^\perp = \{ 0\}$. In particular, we denote by $P_{\ker(\nabla)}:L^q(\Omega) \rightarrow \ker(\nabla)$ the continuous linear projection such that $\range(P_{\ker(\nabla)}) = \ker(\nabla)$ and $\ker(P_{\ker(\nabla)}) = \ker(\nabla)^\perp$.

Having these prerequisites at hand, we now arrive at the main duality result of the paper.

\begin{thm}\label{prop:general_duality_result}
Let $p\in (1,\ds]$, $\lambda > 0$, $S$ a Banach space, and $K \in \mathcal{L}(L^p(\Omega) ,S)$. Further let $D:S \rightarrow \overline{\RR}$ be convex, lower semi-continuous, and continuous and finite at zero.
Also, assume that at least one of the following two conditions hold:
\begin{align*} (i) & \, \left\{ \begin{aligned} 
&\text{There exists } c>0 \text{ such that } c|z| \leq j(x,z)  \text{ for every } z\in\RR^{d} \text{ and for almost every } x\in\om \text{,} \\
&P_{\ker(\nabla)}\bigg( K^* \big[ \bigcup_{\mu \geq 0} \mu \domain((\lambda D)^* ) \big] \bigg)\text{ is a vector space and } 0 \in \domain ((\lambda D)^*).
\end{aligned} \right. 
\\ (ii) &  \, \bigg\{ 
\range(K^*) \text{ is closed, } \ker(K) \text{ is finite dimensional  and } D \text{ is coercive.}
\end{align*}
Then there exists a solution to the primal problem 
\begin{equation}\label{primal_problem}
 \inf_{u\in L^p (\Omega)} J^{**}(u)  + (\lambda D\circ K)(u) \tag{P}
 \end{equation}
as well as to the predual problem
\begin{equation} \label{predual_problem}
\inf_{\substack{p\in W^q_0(\dive;\om) \\ p \in Q}}  (\lambda D\circ K)^*(\di p ) , \tag{pD}
\end{equation}
where $Q = \{g \in W^q_0(\dive;\om)\cap L^\infty(\Omega,\R^d) : \;j^\circ(x,g(x))\leq 1 \text{ for a.e. } x \in \Omega\}$, and to the saddle-point problem 
\begin{equation}\label{saddle_point_problem}
 \inf_{\substack{p \in W^q_0(\dive;\om) \\ p \in Q}} \sup_{u\in L^p (\Omega)} (\di p ,u)  -  (\lambda D\circ K)(u). \tag{sp}
 \end{equation}
Further, these problems all coincide at their optimal values and are equivalent in the sense that the pair $(p,u)$ is a solution to the saddle-point problem \eqref{saddle_point_problem} if and only if $u$ is a solution to \eqref{primal_problem} and $p$ is a solution to \eqref{predual_problem}.
\end{thm}
Before we provide the proof, let us motivate the rather technical assumptions by showing that they exactly reduce to the setting of Proposition \ref{prop:duality_norm_discrepancy} if the discrepancy term $D$ is a norm: In this setting, we have $D(v) = \|v-f\|_S$, which is obviously finite and continuous at zero, and it is coercive in $S$. Furthermore, from \cite[Theorem 4.4.10]{Borwein}, coercivity (i.e., the first ingredient of assumption $(i)$) implies that $B_\epsilon (0) \subset \domain((\lambda D)^ *)$  for some $\epsilon >0$. Hence the union over all positive factors times $\domain((\lambda D)^ *)$ is all of $S^ *$ and since $P_{\ker(\nabla)}$ is linear, the second part of $(i)$ also holds. Thus, the remaining assumptions correspond exactly to the assumptions of Proposition \ref{prop:duality_norm_discrepancy}. We remark that the rather technical assumptions above allow to also cover settings like the one for structural TV-regularized PET reconstruction with positivity constraints.
\begin{proof}[Proof of Theorem \ref{prop:general_duality_result}]
We first show strict duality. Note that convex conjugation is always carried out in the space where the functional is defined. Define $X = W^q_0(\dive;\om) $, $Y = L^q(\Omega) $, the bounded linear operator
$ T:X \rightarrow Y$, $Tp = \di p$, and the functionals
$F:X \rightarrow [0,\infty]$, $F(p) = \I_{Q}(p)$,
and $G:Y \rightarrow [0,\infty]$, $G(u^{\ast}) = (\lambda D\circ K)^ * (u^{\ast})$.
Our goal is to show the following duality relation, which also asserts existence of a minimum for the right-hand side:
\[ \inf_{x\in X} F(x) + G(Tx) = -\min _{y \in Y^*} F^*(-T^*y) + G^*(y).\]
For the right-hand side we get
\begin{align*}
 -\min _{y \in Y^*} F^*(-T^*y) + G^*(y) 
 & = -\min _{u \in L^p(\Omega)} \I^*_Q( -(\di )^* u) + (\lambda D\circ K)(u)  \\
 & = -\min _{u \in L^p(\Omega)} \sup _{z \in W^q_0(\dive;\om)} \{( - \di z,u) - \I_Q(z)\}  + (\lambda D\circ K)(u) \\
 & = - \min _{u \in L^p(\Omega)} J^{**}(u)   + (\lambda D\circ K)(u).
 \end{align*}
To show the duality relation, according to \cite{Attouch86}, it suffices to show that
\begin{equation} \label{eq:duality_condition}
U:= \bigcup _{\mu \geq 0} \mu \bigg[ \domain((\lambda D\circ K){^\ast}) - T\domain(F)\bigg]  =  \bigcup _{\mu \geq 0} \mu \bigg[
K^*\domain( (\lambda D)^*)  - \di(Q) 
\bigg]
 \subset L^q(\Omega)
\end{equation}
is a closed vector space.  Note that the second equality holds true due to Lemma \ref{lem:domain_predual_vectorspace}.
Now first consider the case that assumption $(i)$ holds. We claim that in this case
\[ U =  P_{\ker(\nabla)}\bigg( K^* \big[ \bigcup_{\mu \geq 0} \mu \domain((\lambda D)^* ) \big] \bigg) + \ker(\nabla)^\perp \]
and hence, by assumption $(i)$, $U$ is  a closed vector space for being the sum of a finite dimensional and a closed vector space. It is clear that $U$ is a subset of the right-hand side since  
\begin{align*}
\bigcup _{\mu \geq 0} \mu \bigg[K^*\domain( (\lambda D)^*)  -  \di(Q)\bigg]
&\subset \bigcup_{\mu\geq 0} \mu K^*\domain( (\lambda D)^*)- \bigcup_{\mu\geq 0} \mu \dive(Q)\\
&\subset K^* \big[ \bigcup_{\mu \geq 0} \mu \domain((\lambda D)^* ) \big]+\ker(\nabla)^\perp\\
& \subset P_{\ker(\nabla)}\bigg( K^* \big[ \bigcup_{\mu \geq 0} \mu \domain((\lambda D)^* ) \big] \bigg)
\\ 
&\quad +P_{\ker(\nabla)^\bot}\bigg( K^* \big[ \bigcup_{\mu \geq 0} \mu \domain((\lambda D)^* ) \big] \bigg) +\ker(\nabla)^\perp\\
&\subset P_{\ker(\nabla)}\bigg( K^* \big[ \bigcup_{\mu \geq 0} \mu \domain((\lambda D)^* ) \big] \bigg) + \ker(\nabla)^\perp.
\end{align*}

For the reverse inclusion, we employ a result of \cite{Holler_ictv} which states that for any $\delta > 0$ there is an $\epsilon >0$ such that $B_\epsilon(0) \cap \ker(\nabla)^\perp \subset \{ \di g : g \in W^q_0(\dive;\om), \, |g(x)| \leq \delta \text{ for a.e. }x\in\om\}$. In fact, this was shown in \cite{Holler_ictv} for $p=\ds$, but the extension to $1<p<\ds$ is direct. 

Take any $u =  P_{\ker(\nabla)}(\mu K^*v) + w $ with $\mu \geq 0$, $v \in \domain((\lambda D)^*)$ and $w \in \ker(\nabla)^\perp$. We re-write $u = \mu K^*v  - ( P_{\ker(\nabla)^\perp}(\mu K^*v) - w) = \mu K^*v - \tilde{w}$ with $\tilde{w}\in \ker(\nabla)^\perp$. 
By assumption $(i)$, we get that $J^*(u^*) \leq  c \TV^*(c^{-1}u^*)$ for all $u^* \in L^q(\Omega)$, which means that if $c^{-1}u^*\in \domain(\TV^*)$, i.e., $u^* = \di g_1$ with $\|g_1\|_{L^\infty(\om)} \leq c$ then $u^* \in \domain (J^*) = \dive (Q)$, i.e., $u^{\ast} = \di g_2 $ with $j^\circ(x,g_2(x))\leq 1$. Taking hence $\epsilon $ sufficiently small such that $\epsilon \tilde{w} = \di g$ with $\|g\|_{L^\infty(\om)} \leq c$, we get that $\epsilon \tilde{w}  \in \dive (Q)$. Now define $\eta = \min \{ 1/\mu,\epsilon\}$. Then $\eta \mu v \in \domain((\lambda D)^*)$ by convexity of $(\lambda D)^*$ and assumption $(i)$. Consequently we get
\[ u = \frac{1}{\eta} ( \eta \mu K^*v - \eta \tilde{w}) \in U, \]
as claimed.

Before considering condition $(ii)$, we show that $V:  = \bigcup _{\mu \geq 0}  - \mu \dive (Q)$ is a vector space. To see this, take $\lambda \in \R$ and  $y_1 := -\mu _1T(z_1)$ and $y_2:=-\mu _2T(z_2)$ with $z_1,z_2 \in Q$ and $\mu_1,\mu_2 \geq 0$. We show that $\lambda y_1 + y_2 \in V$. We note that, by  definition, $Q$ is convex and also $\lambda z \in Q$ for $z \in Q$ and $\lambda \in [0,1]$. Further, from assumption (J3) on $j$, one can check that $-z \in Q$ for $z \in Q$. 
Define $\eta = \max\{\mu_1,\mu_2\} \max\{|\lambda|,1\}$. Then $\frac{\lambda\mu_1 z_1}{\eta} \in Q$ and $\frac{\mu_2 z_2}{\eta} \in Q$. Hence by convexity of $Q$, we have
\[ \lambda y_1 + y_2 = -(\lambda \mu_1 Tz_1 + \mu_2 T z_2) = 2\eta \left(- T\left(\frac{\lambda \mu_1z_1 + \mu_2 z_2}{2\eta}\right)\right) \in V .\]

Suppose now condition $(ii)$ holds. 
We will show that in this case $U$ can be written as a sum of $\ker(K)^{\perp}$ and a finite dimensional space.
Indeed, first note that, following \cite[Section 2.4, Example 2]{Brezis}, $\ker(K)^{\perp}$ is of finite codimension and, consequently, 
%
there exists a finite dimensional space $\tilde{U}\subset L^q(\Omega)$ such that $L^{q}(\om)$ is the direct sum of  $\ker(K)^{\perp}$  and $\tilde{U}$. Thus we can define the canonical continuous linear projection onto $\tilde{U}$, $P_{\tilde{U}}: L^{q}(\om) \to \tilde{U}$.
We claim now that 
\[ U = \ker(K)^\perp + P_{\tilde{U}} (V),\]
which is again a closed vector space for being the sum of a closed and a finite dimensional vector space. First we note that by the close range theorem we have $\ker(K) ^\perp = \range(K^*)$, and hence  $U$ is included in the right-hand side since
\begin{align*}
U&\subset \range(K^*) + V\subset \range(K^*)+P_{\ker(K)^{\perp}} (V)+ P_{\tilde{U}} (V)=  \ker(K)^\perp + P_{\tilde{U}} (V).
\end{align*}
To show the other subset inclusion, take $u = K^*v + P_{ \tilde{U}}(-\mu w) $ with $w \in \dive(Q)$. Again we re-write $u = K^* v - P_{\ker(K)^\perp}(-\mu w) - \mu w = K^* \tilde{v} - \mu w$ with $\tilde{v}\in S^{\ast}$. Again by \cite[Theorem 4.4.10]{Borwein}, coercivity of $\lambda D$ implies continuity of $(\lambda D)^*$ at $0$. Hence there exists $\epsilon > 0 $ such that $\overline{B_\epsilon (0)} \subset \domain ((\lambda D)^*)$. Setting $\lambda = \min \{ \epsilon/\|\tilde{v}\|,1/\mu\}$ we can write
\[ u = \frac{1}{\lambda} (K^*( \lambda \tilde{v}) - \lambda \mu w) \in U\]
since $\eta w \in \dive (Q)$ for $\eta \in [0,1]$. This shows strict duality and existence for \eqref{primal_problem} under assumption $(i)$ or $(ii)$. 

Now we show existence for the predual problem.
From the continuity of $D$ at $0$ 
 and since $K$ is bounded, $\lambda D \circ K$ is continuous at $0$ and again by \cite[Theorem 4.4.10]{Borwein} $(\lambda D \circ K)^*$ is coercive in $L^q(\om)$.
We show that the set $Q$ is bounded with respect to the $L^q$-norm. For this purpose, we define $C=2\gamma >0$, where $\gamma$ is the constant from assumption (J2), and observe that, for any $g \in Q$ and almost every $x \in \Omega$,
\[ C^{-1}|g(x)| =  \sup _{ \substack{ \tilde{z}^*\in \R^d \\ |\tilde{z}^*| \leq (C/\gamma) - 1 }} C^{-1}g(x) \cdot \tilde{z}^* \overset{(*)}{\leq} \sup _{ \substack{ \tilde{z}^*\in \R^d \\ j(x,\tilde{z}^*) \leq C }} g(x) \cdot C^{-1}\tilde{z}^*  = \sup _{ \substack{ z^*\in \R^d \\ j(x,z^*) \leq 1 }} g(x) \cdot z^* = j^\circ (x,g(x))  \leq 1.\]
Indeed, since by (J2), $j(x,z) \leq \gamma (1 +|z|)$, $(*)$ holds true for $C>\gamma$ since then, for $|z^*|\leq (C/\gamma) - 1$, $j(x,z^*) \leq \gamma(1+|z^*|) \leq C$. This implies that $\|g\|_{L^\infty(\om)} \leq C$ and hence also $Q$ is bounded in $L^q(\om)$.

Now taking $(p_n)_{n\in\NN}$ an infimizing sequence for the predual problem, we get by boundedness of $Q$ and coercivity of $(\lambda D \circ K)^*$ that there exist $\hat{p} \in L^q(\Omega,\R^d) $ and $w \in L^q(\Omega)$ such that, up to subsequences, $p_n \rightharpoonup \hat{p}$ and $\di(p_n) \rightharpoonup w $. This implies that $p \in W^q(\dive;\om)$ and $\di p = w$. Further we note that $\{ (g,\di g) : g \in W_0^q(\dive;\Omega) \cap Q \}$ is a convex and closed subset of $L^q(\Omega,\R^{d+1})$, hence it is also weakly closed. By weak convergence of $(p_n,\di p_n) $ to $(\hat{p},\dive \hat{p})$ and lower semi-continuity of $(\lambda D \circ K)$ with respect to weak convergence in $L^q(\Omega,\R^d) $ it follows that $\hat{p}$ is a solution to \eqref{predual_problem}.

Finally, \cite[Proposition III.3.1]{Ekeland} guarantees equivalence to the saddle-point problem \eqref{saddle_point_problem} as claimed and the proof is complete.
\end{proof}


This concludes the main existence and duality result of the paper, from which the application to linear inverse problems with norm discrepancy follows as a special case. A second situation we want to consider in more detail is an inverse problem where the measured data describes the physical density of some quantity and is corrupted by Poisson noise. The main application we have in mind for this setting is PET image reconstruction, where the Poisson log-likelihood and positivity constraints are used for data fidelity.

\subsection{A Poisson noise model.} We are now interested in the problem
\begin{equation}\label{eq:min_prob_pet_formal}
\min _{u \in L^p(\Omega) } J^{**}(u) + \lambda \int _\Sigma (Ku) (\sigma)- f(\sigma) \log ((Ku)(\sigma) + c_0(\sigma))\,d\sigma + \I_{[0,\infty)} (u),
\end{equation}
where $K$ is a linear operator (a slightly modified Radon transform), $f\geq 0$ is the given data, $c_0>0$ is an estimate for measurements due to scattering and random events and $\Sigma$ is a subset of $\R^{n_d}$ with some $n_d \geq 1$. The function $\I_{[0,\infty)}$ constrains the unknown to the non-negative reals. Note that in PET imaging with real data, the estimate $c_0$ of scatter and random events is an integral part of the (so-called) reconstruction pipeline, as it describes a non-negligible part of the data. Such an estimate is typically delivered by the scanner software as preprocessing step.

More specifically, in what follows we invoke the following data assumptions:
\begin{equation} \label{AP}\tag{AP} 
\left\{\begin{array}{l@{\,}l} (i) & \operatorname{meas}(\Sigma) <+\infty,\quad f,c_0 \in L^ 1(\Sigma),\,\, f \geq 0, \,\, c_0 > 0,\\[1ex]
                       (ii) & \sigma \mapsto f(\sigma) \log(c_0(\sigma)) \in L^1(\Sigma),\,\,  \sigma  \mapsto \frac{f(\sigma)}{c_0(\sigma)} \in 			      L^{\infty}(\Sigma),\\[1ex]
                       (iii) &K \in \mathcal{L}(L^ p(\Omega),L^ 1(\Sigma)), \quad Ku \geq 0 \text{ whenever } u \geq 0,\\[1ex]
 		      (iv) &\text{the constant functions are contained in } \range(K^*).
\end{array}\right.
\end{equation} 
When reading the above integrability assumptions on the data and the scatter and random events, one has to keep in mind that typically $f$ describes the Radon transform of some density which is supported in the interior of $\Omega$ and $c_0$ some defects that increase with the density, but are present on every measurement line. We believe that in such a context, the posed assumptions are realistic and not very restrictive. Clearly, they are satisfied if we assume $f$ and $c_0$ to be bounded and $c_0$ to be uniformly bounded away from zero, which is often assumed in the literature for the measured data in order to obtain stability results \cite{sawatzky2013tv}. In this respect one has to keep in mind, however, that the Radon transform of any function which is bounded above will converge to zero for measurement lines whose length converges to zero. Hence we believe that there is a benefit in using the assumption \eqref{AP} rather than uniform boundedness away from zero.

Regarding the assumptions on the forward operator $K$, we recall in the following some basic properties of the Radon transform, which in particular show that for such a $K$ assumption \eqref{AP} is fulfilled. In fact, the classical Radon transform $R$ is a bounded linear operator from $L^{1}(\RR^{d})$ to $L^{1}(\mathcal{S}^{d-1}\times \RR)$; see for instance \cite{oberlin1982mapping}. Here $\mathcal{S}^{d-1}\times \RR$ is equipped with the measure $\mu:=\mathcal{H}^{d-1}\lfloor \mathcal{S}^{d-1}\times \mathcal{L}$. In this context, $d\sigma$ denotes integration with respect the measure $\mu$. Define now the bounded linear operator $E: L^{p}(\Omega)\to L^{1}(\RR^{d})$ as the extension-by-zero outside $\Omega$. Then, if $K:=R\circ E$ we have that $K$ is linear, it maps $L^{p}(\om)$ to $L^{1}(\Sigma)$, with $\Sigma$ a bounded subset of $\mathcal{S}^{d-1}\times \RR$, and it is bounded, since for every $u\in L^{p}(\om)$
\[\|Ku\|_{L^{1}(\Sigma)}=\|(R\circ E)(u)\|_{L^{1}(\Sigma)}\le C\|Eu\|_{L^{1}(\RR^{d})}= C\|u\|_{L^{1}(\om)} \leq \tilde{C} \|u\|_{L^{p}(\om)}.\]
According to \cite{hertle1983continuity}, the adjoint $R^{\ast}:L^{\infty}(\Sigma)\to L^{\infty}(\RR^{d})$ of the Radon transform is given by 
\[R^{\ast}v(x)=\int_{\mathcal{S}^{d-1}} v(\theta, \theta\cdot x) \,d\mathcal{H}^{d-1}(\theta),\quad x\in \RR^{d}.\]
One can easily see now that $K^{\ast}: L^{\infty}(\Sigma)\to L^{q}(\om)$ is simply the restriction of $R^{\ast}$ in $\om$, since for every $u\in L^{p}(\om)$ and every $v\in L^{\infty}(\Sigma)$ we have
\[\int_{\om} u R^{\ast} v\, dx=\int_{\RR^{d}}  Eu R^{\ast}v\, dx= \int_{\Sigma} R (Eu) v \,d\sigma=\int_{\Sigma} (Ku) v\, d\sigma. \]
In this case, the constant functions belong to $\mathrm{Rg}(K^{\ast})$.
Moreover, from the definition of the Radon transform it follows immediately that $Ku\ge 0$ for $u\ge 0$ with both inequalities to be understood in the almost everywhere sense. Hence, assumption \eqref{AP} holds true when $K$ is the Radon transform. Note also that the assumption on $\Sigma$ is valid since due to $\om$ being bounded, for every $u\in L^{1}(\om)$, we have that $Ku$ is supported in a fixed compact subset of $\Sigma$.

Now considering the formal minimization problem for Poisson-corrupted data \eqref{eq:min_prob_pet_formal} in view of our general duality result, some issues arise. First of all we need to rigorously define the data discrepancy term as a function from $L^ p(\Omega)$ to the extended reals, and secondly neither the data discrepancy nor the positivity constraint will be continuous in $L^ p(\Omega)$ around $0$. 

The following modification of the data term resolves some of these issues without changing the original problem.
For $f \in [0,\infty)$ and $c_0 \in (0,\infty)$ we define the integrand
\[ l_{f,c_0}(t):= 
\begin{cases} t - f \log (t+c_0) &\text{if } t \geq 0, \\
 -f \log (c_0) + \big(1 - \frac{f}{c_0}\big) t &\text{if } t<0, \end{cases}
 \]
and for $f:\Sigma \rightarrow  [0,\infty)$ and $c_0: \Sigma \rightarrow (0,\infty)$ the functional 
\[
 \dkl :  L^ 1 (\Sigma)  \rightarrow \overline{\RR}, \quad
 v \mapsto \int_\Sigma l_{f(\sigma),c_0(\sigma)}(v(\sigma)) \,d\sigma + L\|v_-\|_{L^{1}(\Sigma)},
\]
where we set $L := \|1-\frac{f}{c_0}\|_\infty + 1$ and $v_- := \min \{v,0\}$ in a pointwise almost everywhere sense.
The point in the definition of $\dkl $ is that we change the original data fidelity only in points where $Ku$ is negative, which, however, can never occur due to the positivity constraint in $u$ and $Ku \geq 0 $ for $u \geq 0$. Hence, when considering the minimization problem
\begin{equation} \label{eq:pet_original_problem}
\inf_{u\in L^ p(\Omega)} J^ {**}(u) + \lambda \dkl(Ku)  + \I_{[0,\infty)}(u),
\end{equation}
it is immediate that the set of optimal solutions is exactly the same as for \eqref{eq:min_prob_pet_formal}.
The modified fidelity $\dkl$ enjoys the following properties.
\begin{lem} \label{lem:kl_properties} Assume that \eqref{AP} holds. Then $\dkl$ is convex and continuous in $L^1(\Sigma)$. Further 
there exist constants $M,N>0$ such that
\[ \|v\|_{L^1(\Sigma)} \leq M \dkl(v) + N \quad \text{for all } v \in L^1(\Sigma).\]

\begin{proof}
Regarding convexity, we note that $v \mapsto \|v_-\|_{L^1(\Sigma)}$ is convex and hence it suffices to show convexity of the integrand $z \mapsto l_{f(\sigma),c_0(\sigma)}(z)$ for $\sigma \in \Sigma $ fixed. To this aim, we note that convexity is equivalent to the derivative of $l_{f(\sigma),c_0(\sigma)}$ being monotonously increasing. The latter is indeed true since for $z>0$ and $z<0$ the integrand is convex and since, as can be readily checked, the left and right derivative of $l_{f(\sigma),c_0(\sigma)}(\cdot)$ at $z=0$ coincide.

 
Regarding continuity, for $v \in L^1(\Sigma)$ we denote $\Sigma_1 = \{ \sigma \in \Sigma : v(\sigma) \geq 0 \}$, $\Sigma_2 = \Sigma \setminus \Sigma_1$ and estimate
\begin{align*}
 \dkl(v) 
 & \leq  \int _{\Sigma_1} v(\sigma) - f(\sigma)\log (c_0(\sigma)) \,d\sigma   + \int _{\Sigma_2}  - f(\sigma)\log ( c_0(\sigma))  + \left(1-\frac{f(\sigma)}{c_0(\sigma)}\right) v(\sigma) \,d\sigma + L \|v_-\|_{L^{1}(\Sigma)}    \\
 & \leq  L\|v \|_{L^{1}(\Sigma)} +  \|f\log(c_0)\|_{L^{1}(\Sigma)}  + \left(1 + \Big\|\frac{f}{c_0}\Big\|_{\infty}\right) \|v_- \|_{L^{1}(\Sigma)}  \\
 &\leq  \left(1 +  \Big\|\frac{f}{c_0}\Big\|_{\infty } + L \right) \|v\|_{L^{1}(\Sigma)}  + \|f\log(c_0)\|_{L^{1}(\Sigma)}  .
 \end{align*}
 Hence $\dkl$ is bounded above in a neighborhood of any point and, for being a convex function \cite[Prop. 4.1.4]{Borwein} it is also continuous at any point.

To show the coercivity estimate, first pick any $v \in L^ 1(\Sigma)$ with $v \leq 0$. Choosing $c =\| 1- \frac{f}{c_0}\|_\infty$ we get
\begin{equation} \label{eq:log_estimate_neg}
\begin{aligned}
 \int _\Sigma -f \log (c_0) + \left(1 - \frac{f}{c_0}\right) v\,d\sigma + L\|v\|_{L^{1}(\Sigma)}  
 & \geq -\|f \log(c_0)\|_{L^{1}(\Sigma)}  - c \|v\|_{L^{1}(\Sigma)} + L \|v\|_{L^{1}(\Sigma)}\\ 
  & = -\|f \log(c_0)\|_{L^{1}(\Sigma)} +  \|v\|_{L^{1}(\Sigma)}  .
 \end{aligned}
 \end{equation}

Further, for $v \in L^1(\Sigma)$ with $v \geq 0$ the Poisson log-likelihood can be estimated below as follows \cite{Borwein91, Resmerita07}:
It is easy to check by differentiating that the function
\[ t \mapsto \left( \frac{4}{3} + \frac{2t}{3} \right) (t \log( t) - t+1) - (t-1)^2,\quad t\ge 0, \] is convex and attains its minimum value of $0$ at $t=1$. Setting $t = \frac{f(\sigma)}{v(\sigma)+c_0(\sigma)}$ we get that
\begin{align*}
 (v(\sigma)+c_0(\sigma) - f(\sigma))^2 
 &\leq \left[ \frac{2}{3}f(\sigma) + \frac{4}{3}(v(\sigma)+c_0(\sigma)) \right] \bigg(f(\sigma)\log(\frac{f(\sigma)}{v(\sigma)+c_0(\sigma)}) - f(\sigma) + v(\sigma)+c_0(\sigma)\bigg) \\
 &=  \left[ \frac{2}{3}f(\sigma) + \frac{4}{3}(v(\sigma)+c_0(\sigma)) \right] \bigg( v(\sigma) - f(\sigma) \log(v(\sigma)+c_0(\sigma)) \\
 & \qquad + c_0(\sigma) - f(\sigma) +f(\sigma)\log(f(\sigma)\bigg).
 \end{align*}
Now taking the square root on both sides, integrating and using the Cauchy-Schwarz inequality we get that
\[ \|v+c_0 - f\|_{L^{1}(\Sigma)} ^2 \leq \left[ \frac{2}{3}\|f\|_{L^{1}(\Sigma)} + \frac{4}{3}\|v+c_0\|_{L^{1}(\Sigma)} \right]\bigg( \int_\Sigma  v - f \log(v+c_0)\,d\sigma +\int_\Sigma   c_0 - f +f\log(f)\,d\sigma \bigg) . \]
Further, using that $\|v+c_0\|_{L^{1}(\Sigma)}^{2} \leq 2 \|v+c_0 - f \|_{L^{1}(\Sigma)}^2 + 2 \|f\|_{L^{1}(\Sigma)}^2 $ and denoting by $M,N$ generic constants with $M>0$ we get
\begin{align*}
\|v+c_0\|_{L^{1}(\Sigma)} \bigg( \|v+c_0\|_{L^{1}(\Sigma)} - \frac{8}{3} \left( \int_\Sigma  v - f \log(v+c_0)\,d\sigma\right) - N \bigg) \leq  M \left( \int_\Sigma  v - f \log(v+c_0)\,d\sigma \right) + N.
\end{align*}
Now in case $\bigg( \|v+c_0\|_{L^{1}(\Sigma)} - \frac{8}{3} \big( \int_\Sigma  v - f \log(v+c_0)\,d\sigma\big) - N \bigg)  \leq 1$ we get
\[ \|v+c_0\|_{L^{1}(\Sigma)}  \leq  \frac{8}{3} \left( \int_\Sigma  u - f \log(u+c_0)\,d\sigma\right) + N +1. \] In the other case we get
\[ \|v+c_0\|_{L^{1}(\Sigma)}  \leq  M \left( \int_\Sigma  v - f \log(v+c_0)\,d\sigma \right) + N.\]
Hence in any case there exists constants $M>0$, $N \in \R$ such that
\begin{equation} \label{eq:log_estimate_pos}
\|v\|_{L^{1}(\Sigma)}  \leq  M \left( \int_\Sigma  v - f \log(v+c_0)\,d\sigma \right) + N.
\end{equation}
Now splitting an arbitrary $v \in L^ 1(\Sigma)$ into its positive and negative parts, and using the estimates \eqref{eq:log_estimate_neg}, \eqref{eq:log_estimate_pos} accordingly, we get that there exists $M>0$, $N \in \R$ such that
\[ \|v\|_{L^{1}(\Sigma)}  \leq  M \dkl(v) + N, \]
which completes the proof.
\end{proof}
 \end{lem}

Using these properties of the data discrepancy, we obtain the following existence result by standard arguments.
\begin{prop}\label{duality_result_poisson}
Assume that \eqref{AP} is fulfilled and there exists $c>0$ such that $c|z| \leq j(x,z)$ for every $z\in\RR^{d}$ and for almost every $x\in\om$.
Then there exists a solution to  
\begin{equation}\label{primal_problem_PET_original}  
 \inf_{u\in L^p (\Omega)} J^{**}(u)  + \lambda \dkl(Ku) + \I_{[0,\infty)}(u). 
 \end{equation}
 \proof We only provide a sketch: Let $(u_n)_{n\in\NN}$ be an infimizing sequence for \eqref{primal_problem_PET_original}.  Since $\ker(K) = \range(K^*)^\perp$ \cite[Corollary 2.18]{Brezis}, $K$ does not vanish on constant functions. Hence, we get from the equivalence of $J^{**}$ to $\tv$ and the coercivity of $\dkl$ (Lemma \ref{lem:kl_properties}) as well as norm equivalence in finite dimensional spaces that $(u_n)_{n\in\NN}$ is bounded in $L^p(\Omega)$. Hence there exists a weakly convergent subsequence and from weak lower semi-continuity of all terms appearing in the objective functional, existence follows readily.
\end{prop}

For applying the general duality result of Theorem \ref{prop:general_duality_result} to the situation of Poisson noise, the only issue that remains is the fact that functional realizing the positivity constraint is not continuous at zero. To overcome this, we will introduce a slight modification of this functional as follows. We define the penalty term
\[ H(u) = M\|u_-\|_{L^{1}(\om)}. \]
The following properties of $H$ are immediate.
\begin{lem} The functional $H:L^p(\Omega) \rightarrow \overline{\RR}$ is convex, lower semi-continuous, bounded above in a neighborhood of zero and $H^*(u^*) = \I_{[-M,0]}(u^*)$.
\end{lem}
Replacing $\I_{[0,\infty)}$ by $H$ obviously is  modification of the original problem for which, in general, we cannot expect that the set of solution coincides with the one of the original problem. However, $H$ is chosen in the spirit of exact penalty functions. For the latter it can be shown that for sufficiently large penalty parameter, the solution of the original problem also solves the penalized one.
We also note that
our goal is to show equivalence of \eqref{eq:pet_original_problem} to a saddle-point problem, which is then numerically solved with a primal-dual algorithm. For the original problem, the positivity constraint would then result in a projection of the unknown to the positive reals in every iteration. The modified objective results in a soft-shrinkage operation for the negative values at each iteration, i.e., $u$ is unchanged at positive values and replaced by $u(x) = \min\{u(x) + M,0\}$ for negative values. 


We finally get the following result for the modified PET problem:

\begin{prop}\label{duality_result_poisson_smoothed}
Assume that (AP) is satisfied,  
 and that there exists $c>0$ such that $c|z| \leq j(x,z)$ for every $z\in\RR^{d}$ and for almost every $x\in\om$.
Then there exists a solution to the primal problem 
\begin{equation}\label{primal_problem_PET}
 \inf_{u\in L^p (\Omega)} J^{**}(u)  + \lambda \dkl(Ku) + H(u), \tag{P-{\footnotesize PET}}
 \end{equation}
to the corresponding predual problem (\emph{pD-{\footnotesize PET}}), and to the saddle-point problem 
\[ \inf_{\substack{p \in W^q_0(\dive;\om)  \\ p \in Q}} \sup_{u\in L^p (\Omega)} (\di p,u)  -   \lambda \dkl(Ku) - H(u). \]
The optimal values of all of these problems coincide, and the problems are equivalent in the sense that a pair $(p,u)$ is a solution to the saddle-point problem if and only if $u$ is a solution to \eqref{primal_problem_PET} and $p$ is a solution to (\emph{pD-{\footnotesize PET}}).
\begin{proof} 
We apply the results of Theorem \ref{prop:general_duality_result} with the data term being $u \mapsto (\lambda \tilde{D}\circ \tilde{K})(u)$, where $\tilde{K}:L^p(\Omega) \rightarrow L^1(\Sigma) \times L^p(\Omega)$ with $\tilde{K}u = (Ku,u)$, and $\tilde{D}:L^1(\Sigma) \times L^p(\Omega) \rightarrow \overline{\R}$, $(v,u) \mapsto \tilde{D}(v,u) = \dkl(v) + \lambda ^{-1} H(u)$. 

Since both $\dkl$ and $H$ are continuous and finite at $0$, also $\tilde{D}$ is continuous and finite at $0$.
It is hence only left to show that the condition $(i)$ of Proposition \ref{prop:general_duality_result} holds true. For this purpose, we  first note that $\domain((\lambda \tilde{D})^*) = \domain ((\lambda \dkl)^*) \times \domain(H^*)$. 
Hence
\begin{align*} 
 \tilde{K}^* \big[ \bigcup_{\mu \geq 0}\mu \domain((\lambda \tilde{D})^* ) \big]  
&=  \tilde{K}^* \big[ \bigcup_{\mu \geq 0} \mu(\domain ( (\lambda \dkl)^*) \times  \domain(H^*)) \big]\\ 
&  \supset \tilde{K}^* \big[ \bigcup_{\mu \geq 0}\mu( B_\epsilon (0) \times \{0\}) \big] \\
&= \range (K^*).
\end{align*}
Now, since the constant functions are contained in $\range (K^*)$, we get that 
\[ P_{\ker(\nabla)}  \bigg( \tilde{K}^* \big[ \bigcup_{\mu \geq 0} \mu\domain((\lambda \tilde{D})^* ) \big] \bigg) = \ker (\nabla) .\]
Further it is easy to see that $(0,0) \in \domain((\lambda \tilde{D})^\ast)$. Hence, all assumptions of Theorem \ref{prop:general_duality_result} hold and the result follows.
\end{proof}
\end{prop}
\begin{rem} \label{rem:standard_stability}
Summarizing, both for norm and Poison-log likelihood discrepancies, in the situation where the underlying inverse problem is indeed ill-posed, coercivity of the integrand as it appears in the definition of the structural TV prior is necessary for showing existence and for establishing our main duality result. As in this situation the structural TV prior is topologically equivalent to isotropic TV, we note that also classical stability and convergence results for vanishing noise levels can be shown by standard techniques. We refer, for instance, to \cite{Hofmann07_nonlinear_tikhonov_banach} for the norm-discrepancy case and to \cite{sawatzky2013tv} for the case of a Poison-log likelihood discrepancy.
\end{rem}

\section{Numerical examples}\label{sec:numerics}
The goal of this section is to exemplary show the numerical realization of two applications of the above general setting. The first one considers a pure denoising problem. This example is included to provide a setting where the inversion of the forward operator, in this case the identity, is well-posed and hence, using the second assumption of Proposition \ref{prop:duality_norm_discrepancy}, the duality result also holds for integrands which vanish on non-trivial sets. The second problem considers the practically more relevant application of MR-guided PET reconstruction and shows how recently proposed regularization approaches \cite{Schramm17_pls,Ehrhardt_PET_prior} can be realized within our framework.

In an abstract setting, the variational problem setting of the previous section can be written as follows:
\[ \min_{u \in L^p(\Omega) } J^{**} (u) + \G(u)  + \Hc(u),\]
where $\G$ and $\Hc$ are smooth and ``simple'' functions, respectively, that are to be specified for the concrete problem setting, e.g., one chooses $\G(u) = \|u-f\|_2 ^2$, $\Hc \equiv 0$ for denoising.
We have shown that, under suitable assumptions, this problem is equivalent to the saddle-point problem
\[\min_{\substack{p \in W^q_0(\dive;\om) \\ p \in Q}} \max_{u\in L^p (\Omega)} (\di p ,u) - \G(u) - \Hc(u).  \]
The main point in this reformulation is that it allows to solve the variational problem without requiring explicit knowledge of $J^{**}$.
We recall that our main results contains a saddle-point formulation whose numerical realization is directly amenable to first-order primal-dual algorithms; see, e.g., \cite{chambolle2011first,chambolle2015ergodic_primal_dual}. In fact, since for $p=2$ the saddle-point problem is defined in a Hilbert space, a corresponding formulation of algorithm can even be shown to be convergent in the infinite dimensional setting \cite{chambolle2015ergodic_primal_dual}. In practice, however, a proper infinite-dimensional version of the algorithm would require the computation of some proximal mappings in $H(\di;\Omega)$ space, which in turn requires the solution of a non-trivial problem (essentially TV denoising) on its own. 

For the derivation of an implementable algorithm, we first discretize and change to the Euclidean norm for both underlying spaces. The iteration steps of the algorithm presented in \cite{chambolle2015ergodic_primal_dual} for solving the saddle-point problem above can then be specified in a general form as follows: Given some triple $(u,p,\bar{p})$, the updates $u_+$, $p_+$, $\bar{p}_+$ are computed as
\begin{equation} \label{eq:pd_algorithm_general_form}
\left\{
\begin{aligned}
u_+ & = \prox_{\sigma,\Hc} \Big(   u + \sigma \di \overline{p} + \sigma  \nabla \G(u) \Big), \\
p_+ & = \proj_Q \Big(   p  + \tau \nabla u_+  \Big), \\
\overline{p}_+ & = 2 p_+ - p.
\end{aligned}
\right.
\end{equation}
Here, $\prox_{\sigma,\Hc}$ denotes the proximal mapping defined as $\prox_{\sigma,\Hc}(u) = \arg\min _{y} \frac{\|u-y\|_2^2}{2\sigma}  + \Hc(y)$, which is well-defined by convexity of $\Hc$. The term $\proj_Q$ is a projection onto the set $Q$ and results from the proximal mapping of $\I_{Q}$, and $\sigma,\tau >0$ need to be suitably chosen. This iteration is then repeated with $(u,p,\bar p):=(u_+,p_+,\bar p_+)$ until some stopping rule is satisfied. 

\subsection{Weighted TV with vanishing weight}

For the case of weighted TV denoising one can simply choose $\Hc\equiv 0$ and $\G(u) := \frac{\lambda}{2}\|u-f\|_{L^{2}(\om)} ^2$, with $f \in L^2(\Omega)$ some given data. This results in the saddle-point problem
\begin{equation}\label{sp_weighted}
\inf_{\substack{p \in W^2_0(\di;\om) \\ p \in Q}} \sup_{u\in L^2 (\Omega)} (\di p ,u) - \frac{\lambda}{2}\|u-f\|_{L^{2}(\om)}^{2},
\end{equation} 
where $Q = \{g \in W^2_0(\di;\om)\cap L^\infty(\Omega,\R^d) : \;|g(x)|\leq \alpha(x) \text{ for a.e. } x \in \Omega\}$.
For a proof of concept, we compute  a weighted TV-based denoising example, with vanishing continuous weight function $\alpha$, i.e., $j(x,z)=\alpha(x)|z|$, with $\alpha\in C(\overline{\om})$, $\alpha\ge 0$. This fits into the regime of condition $(ii)$ of Theorem \ref{prop:general_duality_result}, i.e., $K=id$, hence $\mathrm{Rg}(K^{\ast})$ is closed and $\ker(K)$ is finite dimensional and the fact that $\alpha$ vanishes implies that the coercivity assumption $(i)$ of the same theorem does not hold. Notice that because of this lack of coercivity we cannot expect the solution $u$ to be in $\mathrm{BV}(\Omega)$. Thus, the characterization $J^{\ast\ast}(u)=\int_{\Omega}\alpha\,d|Du|$ is not valid here. This deficiency, however, can be overcome by solving the corresponding saddle-point problem. 

 After a standard discretization procedure the saddle-point problem \eqref{sp_weighted} can be solved straightforwardly with the algorithm in \eqref{eq:pd_algorithm_general_form}, where the involved proximal mappings $\prox_{\sigma,\Hc}$ and $\proj_Q$ can be computed explicitly and pointwise. We use standard forward differences for the discrete gradient with pixel replication at the boundary. The discrete divergence operator is defined by the adjoint relation $\nabla=-\di^{\mathrm{T}}$. We use an accelerated version of the primal-dual algorithm of \cite{chambolle2011first}, as described in \cite{chambolle2015ergodic_primal_dual}, where for the step-sizes $\tau$ and $\sigma$,  the update rule $\tau_{n}=1/(n+1)$, $\sigma_{n}=8/\tau$ is employed, with $\tau_{0}=10^{-4}$ and $\sigma_{0}=8\times 10^{4}$. Here, $n$ is the iteration number and $8$ is an upper bound for the squared norm of the discrete gradient operator, i.e., $\|\nabla\|^2\le 8$. The primal and dual variables were initialized as $u_{0}=f$ and $p_{0}=0$ respectively. As a stopping rule we used a maximum number of iterations $(n=2000)$, after which no considerable change in the iterates was observed.

  Regarding the weight function $\alpha$, we note that our purpose here is not to determine a procedure for automatically selecting it -- for that we refer the reader to \cite{hint_rau, hint_rau_tao_langer}. Nevertheless, in order to produce the weight function $\alpha$ needed for the examples in Figure \ref{fig:weighted} we enable the following procedure: Initially we detect the edges of the noisy image in Figure \ref{fig:noisy} (Gaussian noise of zero mean and  standard deviation $\sigma=0.5$) by using an edge detection algorithm such as Canny algorithm \cite{canny}; compare Figure \ref{fig:edges}. For that we used the MATLAB's built-in Canny function using the parameter values $Thresh=[0.275, 0.55]$ and $Sigma=1.8$.
 
 Then we construct a positive continuous weight function $\alpha$ which takes the value zero exactly on the edges of the image. This can be done by computing the distance function to the edge set, which can be achieved for instance by solving the Eikonal equation \cite{sethian1999level}; see Figure \ref{fig:weight}. The fact that $\alpha$ vanishes exactly at the edge set leads to a piecewise constant reconstruction with less loss of contrast and better edge preservation when compared to the standard scalar TV-regularization; compare Figures \ref{fig:scalarTV} and \ref{fig:weightedTV}. See also \cite{mine_spatial} for a theoretical justification of this experimental observation. 

\begin{figure}[h!]
\begin{center}
\begin{subfigure}[t]{0.3\textwidth}\captionsetup{width=.8\linewidth}
\includegraphics[width=0.95\textwidth]{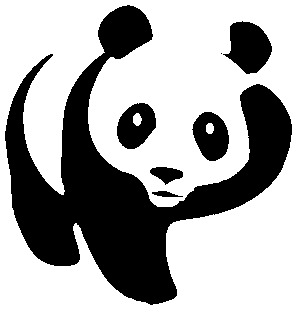}
\caption{Original.}
\end{subfigure}
\begin{subfigure}[t]{0.3\textwidth}\captionsetup{width=.8\linewidth}
\includegraphics[width=0.95\textwidth]{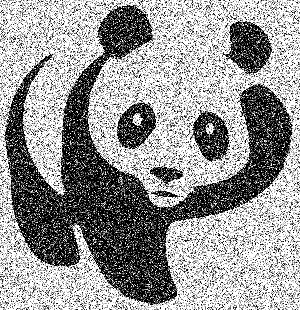}
\caption{Noisy, SSIM=0.1717,\\ MSE=0.0821.}
\label{fig:noisy}
\end{subfigure}
\begin{subfigure}[t]{0.3\textwidth}\captionsetup{width=.8\linewidth}
\includegraphics[width=0.95\textwidth]{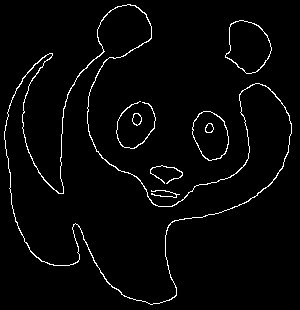}
\caption{Detected edge set.} 
\label{fig:edges}
\end{subfigure}
\vspace{0.2cm}

\begin{subfigure}[t]{0.3\textwidth}\captionsetup{width=.8\linewidth}
\includegraphics[width=0.95\textwidth]{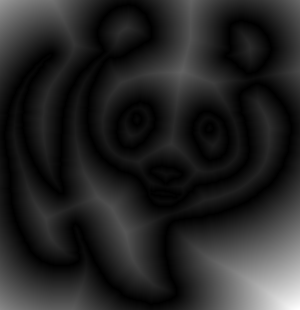}
\caption{Weight function $\alpha$ vanishing on the edge set.}
\label{fig:weight}
\end{subfigure}
\begin{subfigure}[t]{0.3\textwidth}\captionsetup{width=.92\linewidth}
\includegraphics[width=0.95\textwidth]{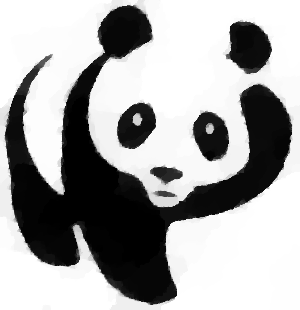}
\caption{Scalar TV, where\\  the scalar parameter $\alpha$ \\is optimized for best SSIM,\\  (SSIM=0.7759, MSE=0.0255).}
\label{fig:scalarTV}
\end{subfigure}
\begin{subfigure}[t]{0.3\textwidth}\captionsetup{width=.8\linewidth}
\includegraphics[width=0.95\textwidth]{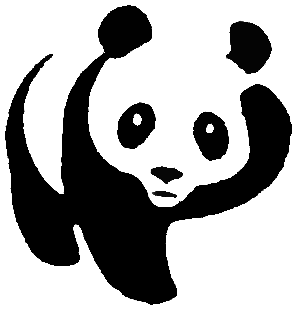}
\caption{Weighted TV using as weight the function $\alpha$ of Figure \ref{fig:weight}, (SSIM=0.9408, MSE=0.0182).}
\label{fig:weightedTV}
\end{subfigure}
\caption{Weighted TV-based denoising via solution of the saddle-point problem \eqref{saddle_point_problem} for $K=id$ and $j(x,z)=\alpha(x)|z|$, with $\alpha\in C(\overline{\om})$, $\alpha\ge 0$.\\ Image source: \url{https://pixabay.com}, (licence Creative Commons CC0).}
\label{fig:weighted}
\end{center}
\end{figure}

\subsection{MR-PET reconstruction}
Magnetic resonance imaging (MRI) and positron emission tomography (PET) are two complementary imaging techniques that are both intensively used in clinical applications. While MRI allows to resolve soft tissue contrast with comparably high spatial resolution, standard MRI techniques deliver qualitative information only. In contrast, PET imaging suffers from a poor spatial resolution, but it is able to deliver quantitative information (the distribution of a radioactively marked tracer). In the past years, and recently particularly motivated by the availability of joint MR-PET scanners, there has been a substantial research effort towards combining MR and PET data in order to obtain benefits for the ill-posed image reconstruction process; see for instance the references provided in \cite{Ehrhardt_PET_prior}. While some works aim at a simultaneous reconstruction of MR and PET \cite{Ehrhardt_PET_MRI, Holler_PET_MRI, Rasch17}, aiming to obtain mutual benefits for both modalities, most approaches employ an MR-based structural prior to improve upon the poor spatial resolution of PET \cite{Ehrhardt_PET_prior,Schramm17_pls, Vunckx2012}. Our work here provides an analytical basis for the latter. We note that our theory covers different priors that were proposed for structure-guided PET reconstruction in \cite{Ehrhardt_PET_prior,Schramm17_pls}. In particular, it allows to conclude that well-posedness of the PET reconstruction problem can only be expected if the prior satisfies the coercivity condition stated in Proposition \ref{duality_result_poisson_smoothed}.

We now exemplary work out the application of our framework to one particular prior, that is motivated by  \cite{Ehrhardt_structure_TV,Schramm17_pls}. Assume that we are given an already reconstructed MR image $v$ defined on a joint image domain $\Omega$. We supppose that $\nabla v \in L^{1}_{\text{loc}}(\Omega,\R^2)$. We then define the structural TV-based prior for PET reconstruction as follows. For parameters $\eta,\nu>0$ with $\eta \in (0,1)$ and $\nu $ small we define
\[ w(x) = \frac{\nabla v(x)}{\sqrt{|\nabla v(x)|^2 + \nu}} ,\]
the matrix-field
\[ A(x) = \sqrt{ I - \eta^2 w(x) \otimes w(x) }, \]
where $\sqrt{\cdot}$ refers to the matrix square root,
and the integrand 
\[ j(x,z) = |A(x)  z | .\]
It is easy to see that $j$ fulfills the assumptions (J1)--(J3). Hence, we can define the structural prior
\[ J(u) = \begin{cases} \int _\Omega |A(x) \nabla u (x) | \,dx &\text{if } u \in W^{1,1}(\Omega), \\ 
+\infty & \text{else,} \end{cases}\]
and consider its lower semi-continuous relaxation for regularization. The effect of this prior can be understood by re-writing
\[ j(x,z) = \sqrt{z^T A^T(x)A(x)z} = \sqrt{|z|^2 - \eta^2(z, w(x))^2} = |z|\sqrt{1 - \eta \big(\frac{z}{|z|},w(x)\big)^2}. \]
This shows that, at any point $x \in \Omega$, the cost of $j(x,z)$ depends on the extent to which $z$ is aligned with $w(x)$. In particular, the smallest cost will appear if $z$ and $w(x)$ are parallel. 
Regarding the parameters $\eta$ and $\nu$, we note that $\nu$ is used to carry out a regularized division in the normalization of $w$, as suggested for instance in \cite{Ehrhardt_PET_prior}. The parameter $\eta \in (0,1)$ does not appear in previous works and we use it to ensure coercivity of $j$, and hence that $J^{**}$ indeed has a regularizing effect. In fact, noting that
\[ j(x,z) = \sqrt{(|z|^2 - (z,w(x))^2)\eta^2 + (1-\eta^2)|z|^2} \geq \sqrt{1-\eta^2}|z|, \]
we obtain coercivity of $j$ and hence, by Proposition \ref{duality_result_poisson_smoothed}, well-posedness of the structural-prior based PET reconstruction problem given as
\begin{equation} \label{eq:pet_application_minprob}
\min _{u \in L^p(\Omega)} J^{**}(u) + \lambda \dkl (Ku) + H(u).
\end{equation}  
Further, equivalence and well-posedness of a corresponding saddle-point problem can be concluded in function space and, due to topological equivalence of $J^{**}$ with $\TV$, standard stability results can be expected; see Remark \ref{rem:standard_stability}.

\noindent \textbf{Algorithmic realization.} In order to derive a numerical algorithm for solving the structural-prior-based PET reconstruction problem in two dimensions, we now replace the image domain $\Omega$ and the data domain $\Sigma$ by a discretized pixel grid, that is, we set $\Omega = \R^{N_\Omega\times M_\Omega}$ and $\Sigma = \R^{N_\Sigma \times M_\Sigma}$ with $N_\Omega,M_\Omega,N_\Sigma,M_\Sigma \in \N$. For simplicity, we keep the notation of the continuous setting, where now all involved operators, integrals and norms are replaced by their discrete counterparts, using standard discretizations. In particular, for the forward operator for PET, we use the implementation of \cite{koesters_emrecon}, which is designed to work with real scanner data and includes resolution modeling by a convolution with a Gaussian kernel. 

In order to carry out the iteration steps as described in \eqref{eq:pd_algorithm_general_form} and obtain a convergent algorithm in the discretized setting, we need that the derivative of the data term, on which we perform an explicit descent step, is Lipschitz continuous. To achieve this, we carry out a $C^2$ instead of a $C^1$ interpolation of the Poisson-log-likelihood in regions with negative data. In fact, for $f \in [0,\infty)$ and $c_0 \in (0,\infty)$, we define the modified integrand
\[ \tilde{l}_{f,c_0}(t):= 
\begin{cases} t - f \log (t+c_0) &\text{if } t \geq 0, \\
 -f \log (c_0) + \big(1 - \frac{f}{c_0}\big) t + \frac{f}{2c_0^2}t^2 &\text{if } t<0, \end{cases}
 \]
 and a modified data term
\[ \widetilde{\dkl}(v) :=  \int_\Sigma \tilde{l}_{f(\sigma),c_0(\sigma)}(v(\sigma)) \,d\sigma + L\int_\Sigma g(v(\sigma)) \,d\sigma ,\]
where, for a small parameter $\epsilon>0$, 
$g$ is a smoothed 1-norm for the negative values given by
\[g(x) = \begin{cases}
0 &\text{if } x \geq 0, \\
 -\frac{x^3}{\epsilon^2} - \frac{x^4}{2\epsilon^3} &\text{if } x \in [-\epsilon,0], \\
|x|+ \frac{\epsilon}{2} &\text{if } x \leq  -\epsilon. 

\end{cases}
\]
In this context, it is important to remember that any change of $\dkl$ that only affects negative values of $v$ does not change the optimal solutions as long as $Au$ is constrained to be positive at every point.

In order to simplify the computation of $\proj_Q$ in the iteration \eqref{eq:pd_algorithm_general_form}, we further carry out a change of variables in the saddle-point reformulation of \eqref{eq:pet_application_minprob} as $p = A^Tq$, where $A^Tq$ is defined as pointwise matrix-vector multiplication $A^Tq(x) = A^T(x) q(x)$. The resulting saddle-point problem then reads 
\[ \min_{q:\|q\|_\infty \leq 1} \max _u (\dive A^Tq,u) - \lambda \widetilde{\dkl}(Ku) - H(u), \] and the iteration steps in  \eqref{eq:pd_algorithm_general_form} can be stated in explicit form as
\begin{equation} \label{eq:pd_algorithm_pet}
\left\{
\begin{aligned}
u_+ & = \prox_{\sigma,H} \Big(   u + \sigma \di A^T \overline{q} - \sigma \lambda  \nabla \widetilde{\dkl}(u) \Big), \\
q_+ & = \proj_{\{\|\cdot \|\leq 1\}} \Big(   q  + \tau A\nabla u_+  \Big), \\
\overline{q}_+ & = 2 q_+ - q,
\end{aligned}
\right.
\end{equation}
where $(\prox_{\sigma,H} (u))(x)= \max\{u(x),\min \{ u(x)+\sigma M,0\}\} $ is a soft-thresholding operator and \\ $(\proj_{\{|\cdot |_2\leq 1\}}(q))(x) = q(x)/\max\{ 1,|q(x)|_2\}$. Following \cite{chambolle2015ergodic_primal_dual}, convergence of the algorithm can be guaranteed under the step-size constraints $\frac{1}{\tau}(\frac{1}{\sigma}-L) \geq \|\dive K^T\|$. Hence, we analytically estimate both $\|\dive A^T\|$ and $L$, with the latter being the Lipschitz constant of the derivative of $\lambda \widetilde{\dkl}$. However, in order to accelerate convergence in practice, we multiply the analytical estimate of $L$ by $10^{-3}$, which increases the admissible step-size. Even tough convergence naturally cannot be guaranteed for this increase step-size choice, we did not experience any convergence issue in practice, which might indicate that our analytical estimate of $L$ are too conservative.

To evaluate the effect of structural coupling compared to standard TV-regularization, we simulated a 2D PET and MR scan using a slightly modified version of the brain phantom of \cite{Aubert_broche_phantom}; see Figure \ref{fig:pet_mr_data}. In fact, we used the software provided in \cite{Aubert_broche_phantom} to simulate an MR image with MPRAGE contrast and a FDG-PET image. Both images share similar structures, but they also contain separate features. In addition, a separate lesion was added to each image, in the top right area for the PET image and in the top left area for the MR image, see corresponding red squares, (see \cite{Holler_PET_MRI} for more details on quantitative values for the PET phantom). Further, to avoid a bias resulting from piecewise constancy, we added two different linear grayscale gradients in the non-background region of both images. 

The MR image was then used as structural prior to define the regularization term $J$ as above. PET data, denoted above by $f$, was generated by forward-projecting the PET phantom, adding simulated random and scatter events (denoted above by $c_0$) and adding multiplicative Poisson-noise. Two different noise levels (strong and medium) were simulated. For comparison, we also carried out a standard TV-regularized reconstruction. The reconstructions for both methods were obtained by iterating the steps provided in \eqref{eq:pd_algorithm_pet}, where for standard TV-regularization $A$ was replaced by the identity. In order to ensure optimality, we carried out $10^4$ iterations for each method, while in practice about 1000 iterations seem sufficient to obtain close-to-optimal results. For both TV- and structural TV-regularization we tested a range of different regularization parameters ($\lambda$) and provide here the results with the lowest mean-squared error.

The resulting images can be seen in Figure \ref{fig:pet_recon_results}. As one would expect, for both noise levels, the structural-prior based reconstruction is able to recover sharper edges whenever they are aligned with the prior image. This results also in improved MSE values. In the PET only lesion, it can be observed that both TV and structural TV obtain similar results. In the region of the PET image where the MR prior has an additional feature, no particular feature transfer can be observed. This indicates a certain robustness of the approach with respect to non-aligned features between the prior and the reconstructed image. We emphasize again that the results presented here should be understood as a proof-of-concept only, and we refer to \cite{Ehrhardt_PET_prior, Schramm17_pls} for a detailed evaluation of this kind of prior for PET reconstruction, also on 3D PET patient data.

\begin{figure}
\begin{tikzpicture}
\draw (0,0) node{\includegraphics[width=0.39\textwidth]{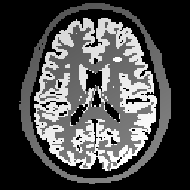}};
\draw[red] (0.25,0.7) -- (1.25,0.7) -- (1.25,1.7) -- (0.25,1.7) -- (0.25,0.7);
\draw (0.4\textwidth,0) node{\includegraphics[width=0.39\textwidth]{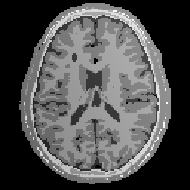}};
\draw[red] (5.44,0.74) -- (6.44,0.74) -- (6.44,1.74) -- (5.44,1.74) -- (5.44,0.74);
\end{tikzpicture}
\caption{\label{fig:pet_mr_data}Left: Ground truth PET image. Right: MRI MPRAGE contrast image used as structural prior. Note that, in addition to the phantom from \cite{Aubert_broche_phantom}, we have added two separate lesions (PET: Top right, MRI: Top left) and a linear gray-value gradient in non-background regions (PET: Increasing from top left to bottom right, MRI: Increasing from bottom left to top right).}
\end{figure}

\begin{figure}
\begin{tikzpicture}
\draw (0,0) node{\includegraphics[width=0.39\textwidth]{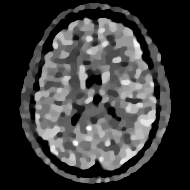}};
\draw (0.4\textwidth,0) node{\includegraphics[width=0.39\textwidth]{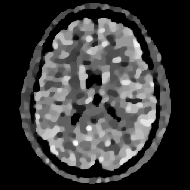}};
\draw (0.705\textwidth,0.132\textwidth) node{\includegraphics[trim=2.7cm 2.2cm 2.5cm 3.5cm, clip, width=0.19\textwidth]{pet_gt.png}};
\draw (0.705\textwidth,0\textwidth) node{\includegraphics[trim=2.7cm 2.2cm 2.5cm 3.5cm, clip, width=0.19\textwidth]{noprior_m50.png}};
\draw (0.705\textwidth,-0.132\textwidth) node{\includegraphics[trim=2.7cm 2.2cm 2.5cm 3.5cm, clip, width=0.19\textwidth]{prior_m50.png}};
\draw (0,-0.4\textwidth) node{\includegraphics[width=0.39\textwidth]{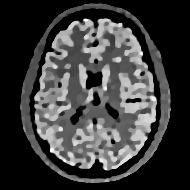}};
\draw (0.4\textwidth,-0.4\textwidth) node{\includegraphics[width=0.39\textwidth]{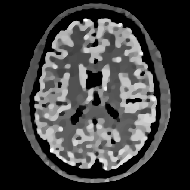}};
\draw (0.705\textwidth,0.132\textwidth-0.4\textwidth) node{\includegraphics[trim=2.7cm 2.2cm 2.5cm 3.5cm, clip, width=0.19\textwidth]{pet_gt.png}};
\draw (0.705\textwidth,0\textwidth-0.4\textwidth) node{\includegraphics[trim=2.7cm 2.2cm 2.5cm 3.5cm, clip, width=0.19\textwidth]{noprior_m200.png}};
\draw (0.705\textwidth,-0.132\textwidth-0.4\textwidth) node{\includegraphics[trim=2.7cm 2.2cm 2.5cm 3.5cm, clip, width=0.19\textwidth]{prior_m200.png}};
\end{tikzpicture}
\caption{\label{fig:pet_recon_results} Standard TV versus structural TV-reconstruction. Top: Strong noise, Bottom: Medium noise. Left: TV regularized reconstruction, Middle: Structural-TV regularized reconstruction, Right: Close up views: Ground truth, standard TV and structural TV (from top to bottom). Mean-squared errors are as follows. Top: 0.3992 (standard TV), 0.3813 (structural TV). Bottom: 0.3445 (standard TV), 0.3209 (structural TV).}
\end{figure}
\section{Conclusion}
In this work, we have analyzed a particular class of structural-prior-based regularization approaches for linear inverse problems in function space. While these approaches have recently been successfully applied to several practically relevant inverse problem settings, an analysis in function space was still missing. Our results show that, in function space, a certain regularity of the prior information in space is necessary in order to obtain an explicit representation of the regularizer. To account for this fact, we have shown how the original minimization problem can alternatively be solved using a saddle-point reformulation that does not require explicit knowledge of the prior. In that context, our main message is that coercivity of the prior is necessary in order to obtain well-posedness (and, consequently, stability by standard results) in non-trivial inverse problem settings. Ultimately, we have shown how to numerically solve the proposed saddle-point reformulation for two proof-of-concept applications including the practically relevant application of structure-guided PET reconstruction.

\subsection*{Acknowledgements}
M.Hinterm\"uller and K.P. acknowledge the support of the Einstein Foundation Berlin within the ECMath project CH12, as well as by the DFG under grant no. HI 1466/7-1 ``Free Boundary Problems and Level Set Methods'' and SFB/TRR154. K.P. acknowledges the financial support of Alexander von Humboldt Foundation. M. Holler acknowledges support of the Austrian Science Fund (FWF) under grant no. P 29192. All the authors would like to thank the Isaac Newton Institute for Mathematical Sciences for support and hospitality during the program ``Variational Methods and Effective Algorithms for Imaging and Vision'' when work on this paper was undertaken. This work was supported by EPSRC grant number EP/K032208/1.

\appendix
\section{}\label{sec:app}

\begin{proof}[Proof of Proposition \ref{pro:equiv_W0_div}.]
It follows by definition of $W^{1,p}(\Omega)$ and $W_{0}^{q}(\di;\Omega)$
that $W_{0}^{q}(\di;\Omega)\subset\ker(\tau)$. In order to
prove $\ker(\tau)\subset W_{0}^{q}(\di;\Omega)$ we show that
$C_{c}^{\infty}(\Omega,\mathbb{R}^{d})$ is dense in $\ker(\tau)$.
Based on an idea according to \cite[Theorem I.2.6]{Girault}, this is done
by showing that any functional $L\in\ker(\tau)^{*}$ vanishing
on $C_{c}^{\infty}(\Omega,\mathbb{R}^{d})$ is zero on $\ker(\tau)$. So let $L\in\ker(\tau)^{*}$
with $L(\phi)=0$ for all $\phi\in C_{c}^{\infty}(\Omega,\mathbb{R}^{d})$.
Since $\ker(\tau)$ is a linear subspace of $W^{q}(\di;\Omega)$,
there exists by the Hahn-Banach extension theorem a continuous extension
$\overline{L}\in W^{q}(\di;\Omega)^{*}$.
Similarly to the space $H(\di;\om)$ \cite{Girault}, we can show that
there thus exist $l=(l_{1},...,l_{d+1})\in L^{p}(\Omega,\mathbb{R}^{d+1})$
such that \[
\overline{L}(g)=\sum_{i=1}^{d}\int_{\Omega}l_{i}g_{i}\,dx+\int_{\Omega}l_{d+1}\di g\,dx\]
for $g=(g_{1},...,g_{d})\in W^{q}(\di;\Omega)$. Now, for $\phi\in C_{c}^{\infty}(\Omega,\mathbb{R}^{d})$
we have\[
0=\overline{L}(\phi)=\sum_{i=1}^{d}\int_{\Omega}l_{i}\phi_{i}\,dx+\int_{\Omega}l_{d+1}\di\phi\,dx\]
from which we conclude that $l_{d+1}\in W^{1,q}(\Omega)$ and $\nabla l_{d+1}=(l_{1},...,l_{d})$.
Thus we can apply the Gauss-Green formula for $W^{q}(\di;\om)$ to $l_{d+1}$
and $h\in\ker(\tau)$ arbitrary and get:\[
0=\int_{\Omega}(\nabla l_{d+1},h)\,dx+\int_{\Omega}l_{d+1}\di h\,dx=L(h)\]
which concludes the proof. 
\end{proof}

\begin{proof}[Proof of Proposition \ref{lbl:Tva_smooth}.]
Note first that since $u\in C^{1}(\om)$, a simple integration by parts yields
\[\tv_{\alpha}^{C}(u)=\sup \left\{\int_{\om} \nabla u\cdot \phi\,dx : \phi\in C_{c}^{1}(\om,\RR^{d}),\; |\phi(x)|\le \alpha(x) \text{ for all }x\in\om \right \}.\]
This shows that $\tv_{\alpha}^{C}(u)\le \int_{\om} \alpha |\nabla u|\,dx$
and we are left to show
$\int_{\om} \alpha |\nabla u|\,dx\le\tv_{\alpha}^{C}(u).$
In order to do so, it suffices to prove that for every $\delta>0$
\begin{equation}\label{integral_delta_less_tva}
\int_{\om_{\delta}} \alpha |\nabla u|\,dx\le\tv_{\alpha}^{C}(u).
\end{equation}
where $\om_{\delta}:=\{x\in\om : \mathrm{dist}(x,\partial \om)>\delta\}$.
Note that since $u\in C^{1}(\om)$, for every $\delta>0$, $\int_{\om_{\delta}} \alpha |\nabla u|\,dx<+\infty$ and hence $u\in \bv(\om_{\delta})$ with $Du=\nabla u \mathcal{L}^{d}\lfloor\om_{\delta}$. Now
if 
\[V_{\alpha}(u,\om_{\delta})=\sup\left \{ \int_{\om_\delta} u\,\di \phi\, dx : \phi\in C_{c}^{\infty}(\om_{\delta},\RR^{d}),\;|\phi(x)|\le \alpha (x),\text{ for all }x\in \om_{\delta} \right\},\] we have that
$V_{\alpha}(u,\om_{\delta})\le \tv_{\alpha}^{C}(u)$, and in order to establish \eqref{integral_delta_less_tva}, it suffices to show that
\begin{equation}\label{equality_delta}
\int_{\om_{\delta}} \alpha |\nabla u|\,dx=V_{\alpha}(u,\om_{\delta}).
\end{equation}
This follows directly from Proposition \ref{lbl:dual_weighted_TV} and the fact that $u\in \bv(\om_{\delta})$. 

\end{proof}

\bibliographystyle{amsplain}

\bibliography{kostasbib,lit_dat}
\end{document}